\documentclass{amsart}
\usepackage[utf8]{inputenc}
\usepackage{hyperref}
\usepackage{amsfonts}
\usepackage{enumerate}
\usepackage{amsthm}
\usepackage{amsmath}
\usepackage{amssymb}
\usepackage{mhequ}
\usepackage{mhsymb}
\usepackage{mhfig}
\usepackage{mathrsfs}  
\usepackage{fullpage}
\usepackage{tikz}
\usepackage{comment}
\numberwithin{equation}{section}

\def\D{\mathscr{D}}

\def\Cr{\mathscr{C}}  

\def\d{\partial}

\def\bX{{\symb X}}

\def\XX{{\mathbb X}}

\def\bR{{\mathbb R}}

\def\OO{{\mathrm O}}

\def\E{{\symb E}}

\def\cD{{\D}}

\theoremstyle{definition}
\newtheorem{definition}{Definition}[section]

\theoremstyle{plain}
\newtheorem{theorem}[definition]{Theorem}
\newtheorem{proposition}[definition]{Proposition}
\newtheorem{lemma}[definition]{Lemma}
\newtheorem{corollary}[definition]{Corollary}
\theoremstyle{remark}
\newtheorem{remark}[definition]{Remark}

\title{Singular paths spaces and applications}
\date{}
\author{Carlo Bellingeri, Peter K. Friz, Máté Gerencsér}
\begin{document}
\maketitle
{{\bf Abstract:} Motivated by recent applications in rough volatility and regularity structures, notably the notion of singular modelled distribution, we study paths, rough paths and related objects with a quantified singularity at zero. In a pure path setting this allows us to leverage on existing SLE Besov estimates to see that SLE traces takes values in a singular H\"older space, which quantifies a well-known boundary effect in the regime $\kappa < 1$. We then consider the integration theory against singular rough paths and some extensions thereof. This gives a method to reconcile, from a regularity structure point of view, different singular kernels used to construct (fractional) rough volatility models and an effective redcoloruction to the stationary case which is crucial to apply general renormalisation methods.}

\section{Introduction}
In classical analysis, the improper integral $ \int_{0+}^1 f(t) dt := \lim_{\eps \downarrow 0} \int_{\eps}^1 f(t) dt$ is studied, with application in many parts of pure and applied mathematics.  It is a basic observation of this paper that this extends to improper {\em Young} and then {\em rough} integrals, as well as recent ramifications to {\em rough volatility} within regularity structures \cite{bayer2019regularity}, \cite[Ch.14]{FH2020}, which provided the original motivation for this article. The notion of {\em singular controlled rough path}, in a H\"older setting, is seen to be consistent with Hairer's {\em singular modelled distributions} \cite{Hairer2014}. We give a self-contained presentation of these spaces including a new time-change characterisation. Moreover, we will introduce the notion of \emph{singular rough paths}, which can be translated in the language of regularity structures as a ``singular model" defined over the \emph{rough path regularity structure}, see \cite[Ch. 13]{FH2020}.
\medskip

Our first application of such singular spaces comes from the seemingly remote field of SLE theory. More specifically, the estimates recently obtained in \cite{friz2017regularity} are seen to imply also singular Besov regularity (w.r.t. to half-space capacity parametrisation) of (chordal) SLE trace. {\color{black}We briefly recall that the classical theory of Loewner evolution gives a one-to-one correspondence between scalar continuous real-valued functions $(\xi_t)_{t\geq 0}$ and families of continuously growing compact hulls $(K_t)_{t\geq 0}$ in the complex upper half-plane $\mathbb{H}=\{z\in \mathbb{C}\colon \Im(z) > 0\}$. There is much interest in the case where these sets admit a continuous trace: i.e. there exists a continuous curve $(\gamma_{t})_{t\geq 0} \subset \overline{\mathbb{H}}$ such that, for all $t>0$, $ \mathbb{H}\setminus K_t$ is the unique unbounded component of $\mathbb{H}\setminus  \gamma([0, t])$ (or even $K_t= \gamma([0,t])$ with $\gamma$ a simple curve). The famous Rohde–Schramm Theorem \cite{rohde_schramm05} asserts that Brownian motion with diffusivity $\kappa\neq 8$  a.s. gives rise to a continuous trace, which is simple when $\kappa \leq 4$. This process is  known as (chordal) SLE trace, denoted by $\gamma_k$. Even for the smoothest possible driver $\xi_t\equiv 0$, with explicitly know trace $ 2i\sqrt{t}$, this function is smooth away from zero. This effect remains visible for $\gamma_k$, for $\kappa\leq 1$, which enjoys H\"older regularity of exponent $\alpha> 1/2$ on compacts away from zero, but (trivially) of exponent $1/2$ on intervals $[0,T]$. We will refer to this phenomenon as \emph{the boundary effect of SLE at $t=0+$}.} Upon extension of classical embeddings to the singular case, we can use  singular H\"older spaces to quantify this boundary effect. {\color{black} Using the notion of $\mathcal{C}^{\alpha, \eta}((0,T])$ space in  Definition \ref{Defn_sing_Holder}, we obtain  when $\kappa \in (0,1)$ the following  a.s. regularity result }
\begin{equation} \label{equ:us}
\gamma_\kappa \in \CC^{\alpha_*^-,(\frac{1}{2})^{-}}((0,1];\mathbb{H})\,,
\end{equation}
 where $\alpha_*^-$ and $(1/2)^{-}$ should be read as any $0<\eta<1/2 $ and any $\alpha < \alpha_* = \alpha_*  (\kappa)$ is given by 
\[\alpha_{\ast}\left( \kappa \right) =1-\frac{\kappa }{24+2\kappa -8\sqrt{\kappa +8}}\,. 
\]
{\color{black}See Theorem \ref{thm_SLE}.} This refines classical results of Lind \cite{Lind2008} and Johansson Viklund{\color{black},} Lawler \cite{JVL11}
\begin{equation} \label{equ:LVL}
\gamma_\kappa 
\in 
\CC^{(\alpha_*^- \wedge \frac{1}{2})}([0,1],\mathbb{H}) 
\cap 
\CC^{\alpha_*^-}([\eps,1],\mathbb{H}).
\end{equation}
By our characterization of $\CC^{\alpha,\eta}$ spaces,  from \eqref{equ:us} we also see that, for $\kappa \in (0,1)$,
\begin{equation} \label{equ:us2}
t \mapsto \gamma_\kappa (t^{2})   \in \CC^{\alpha_*^{-}}([0,1],\mathbb{H}) \;.
\end{equation}
See Corollary \ref{cor_SLE}. Noting that $\alpha_* (\kappa)$ decreases from $1$ to $1/2$, as $\kappa$ ranges from $0$ to $1$, this constitutes a natural SLE extension of know deterministic results (\cite{Friz2016,Shekhar2019} and the references therein) that  $\{ t \mapsto \gamma (t^2) \} \in \CC^1([0,1];\mathbb{H})$, for traces $\gamma$ driven by classes of finite variation paths {\color{black} with some additional properties, like finite energy or being locally regular. An additional heuristic reason behind \eqref{equ:us2} is also related to parametrisation of $\gamma_{\kappa}$, the half-plane capacity, which should behave heuristically like the square root of the intrinsic length, see \cite{JVL11}. Then the reparametrization $t \mapsto \gamma_{\kappa} (t^2) $ can be interpreted as a "parametrisation by arc length``, which absorb the boundary effect.} We then see that the  SLE trace gives rise to (singular) rough paths, in the spirit of \cite{werness2012, horatioSLE}, but insisting on the finer H\"older scale (instead of $p$-variation).

\medskip
Our second application deals with an aspect of fractional stochastic modelling. In particular, we want to reconcile  ``similar'' definitions of a fractional Brownian motion in a rough analysis perspective.  Loosely speaking, we consider the process
$$ 
W^H = K^H {\color{black}*}\, \xi,
$$ 
where $K^H$ is a singular kernel satisfying $K^H(x)\sim|x|^{H-1/2}${\color{black},} $\xi=dW$ is a real-valued white noise {\color{black} and $*$ is the operation of convolution}. To regain analytic stability in the rough path sense, we need to enrich $W^H$ with  It\^o objects like $\int (W^H) dW$ or equivalently $W^H \xi$. This construction has been useful in the analysis of rough volatility models \cite{bayer2019regularity,friz2018precise}, where the modelling (e.g. \cite{bayer2016pricing}) imposes a Riemann--Liouville interpretation of $W^H$, i.e. as left-hand side of
\begin{equation}\label{two_noises}
 \int_0^t |t-s|^{H-1/2} dW_s\quad \text{vs.}   \quad \int_{-\infty}^t K^H(t-s) dW_s \, .
\end{equation}
However, when one applies general results in the theory of regularity structures, stationarity of the noise is important, as is the assumption of a 
compactly supported kernel $K^H$ with the right fractional behaviour near $0$, which points to a different interpretation of $W^H$, namely the right-hand side of \eqref{two_noises}. As a prototypical example, we consider the Wong--Zakai approximation of the random distribution $W^H \xi$, obtained from mollified noise $\xi^\eps = \xi * \varrho_\eps$, {\color{black} where $\varrho_{\eps}(t)= \eps^{-1}\varrho(\eps^{-1}t)$ for some smooth, compactly supported  function $\varrho\colon \bR\to \bR$ such that $\int \varrho =1$ and $\varepsilon >0$}. If we interpret $W^H$  as a Riemann--Liouville fBM, the convergence in a rough (model) topology of the quantity $W^{\eps,H} \xi^\eps$, suitably renormalised, was done by hand in \cite{bayer2019regularity}, also for non-stationary wavelet approximations. On the other hand, the right-hand interpretation of \eqref{two_noises} allows for the use of the general BPHZ renormalisation theorem \cite{ch2016analytic} which in the present setting asserts that $W^{\eps,H} \xi^\eps$ converges upon subtracting its (constant) mean.
 
As it turns out, singular path spaces allow regarding the (non-stationary) left-hand side as singularly controlled by the right-hand side, thereby providing an elegant way to reconcile the difference in \eqref{two_noises}. A related construction appeared in recent work in the setting of singular SPDEs with boundary \cite{mate2018}. Thanks to singular rough integration (reconstruction), statements like
\[
\int_0^t f (W^{\eps,H}_s) \xi^{\eps}_sds - c^{\eps,H} \int_0^t Df (W^{\eps,H}_s) ds \to \int_0^t f ({W}^H_s) d W_s \;, 
\]
with diverging It\^o-Stratonovich correction, are conveniently obtained, see Theorem \ref{final_thm}. We restrict our analysis to $H\in(1/4,1/2)$. This is only important to the extent that  $W^H \xi$ is ``enough'' to make rough analysis work, otherwise branching objects of the form $(W^H)^k \xi$, for $k=1,2,...,K(H)$, are required for robust integration. Since these considerations are well-known and somewhat orthogonal to the ``singular'' theme of this note, we have not pushed for generality in this direction. Similarly, as noted in \cite{bayer2019regularity}, (possibly multidimensional) It\^o--Volterra equations of the form
\begin{equation} \label{equ:IV}
    Y_t= Y_0 + \int_0^t |t-s|^{H-1/2}  F(Y_s) dW_s  
\end{equation}
admit a rough solution theory, by a fixed point arguments in controlled rough path (modelled distribution) space, once the underlying rough paths (model) has been constructed. (This step does not require any stationarity of the noise{\color{black}).}
Wong-Zakai approximation again requires renormalisation, and for appealing to general results \cite{bruned2019rough,ActionRHS} stationarity is again crucial. The same approach then applies, with \eqref{equ:IV} solved in a space {\em singular} controlled rough path (singular modelled distribution), relative to a 
{\em stationary} model that fits, as before, in the BPHZ setting of \cite{ch2016analytic}. For instance, in case $H > 1/4$ this reasoning shows {\color{black} that  the sequence  of solutions} $Y^\eps$
$$
Y^\eps_t= Y^\eps_0 + \int_0^t  |t-s|^{H-1/2} \left( F(Y^\eps_s) \xi^\eps _s  -  c^{\eps,H} F(Y^\eps_s)F'(Y^\eps_s) \right )ds\,,
$$
{\color{black} converges to $Y$ in probability uniformly on compact sets.} Singular spaces are thus seen to provide an approximation theory for It\^o-Volterra problems, commonly used in fractional modelling, as bona fide perturbations of stationary noise problems in the setting of regularity structures for which general results are available.

\section{Singular paths and spaces} \label{sec:SPS}
\subsection{H\"older path spaces}
We are interested in the following generalisation of the usual H\"older spaces. 

\begin{definition}\label{Defn_sing_Holder}
Let $0 < \alpha <  1$ and $\eta \le \alpha$. We say that a continuous path $Y\colon (0,T]\to \R$ is $(\alpha, \eta)$-H\"older, in symbols $Y \in \CC^{\alpha, \eta}((0,T])$  if it satisfies
\begin{equation}\label{def_alphaeta_holder}
\| Y \|_{\alpha,\eta} \eqdef \sup_{0<s<t\leq T} \frac{|Y_t - Y_s|}{s^{\eta-\alpha}|t-s|^{\alpha}} < \infty\,.
\end{equation}
\end{definition}
This is a Banach space with norm given by $|Y_T| + \| Y \|_{\alpha,\eta}$. The classical H\"older space $\CC^{\alpha}([0,T])$ can be identified with $\CC^{\alpha, \alpha}((0,T])$.  The regime $\eta > \alpha$ amounts to enforce a certain vanishing at $0$ and is of no interest to us. We have the following characterisation of $\CC^{\alpha, \eta}((0,T])$ in terms of $\|\cdot\|_{\alpha; I}$ the H\"older semi-norm  on the interval $I$.
\begin{proposition}\label{prop:char_alphaeta_Holder}
An $\alpha$-H\"older path $Y$ belongs to $\CC^{\alpha,\eta}((0,T])$ if and only if $\| Y \|_{\alpha; [\eps,T]} = \OO( \eps^{\eta-\alpha})$   as $\varepsilon\to 0^+$.
Moreover one has the identity
\begin{equation}\label{prop:ident_norms}
\| Y \|_{\alpha,\eta}=\sup_{0<\eps\leq  T}\| Y \|_{\alpha; [\eps,T]}\eps^{\alpha- \eta}\,.
\end{equation}
\end{proposition}

\begin{proof}
Let $ Y\in \CC^{\alpha, \eta}((0,T])$. By definition of $\CC^{\alpha, \eta}((0,T])$ and the monotonicity of $s^{\eta- \alpha}$, for every $0<\eps<1$ and every $\eps\le s<t<T $ one has 
\[|Y_t-Y_s|\leq \| Y \|_{\alpha,\eta} s^{\eta- \alpha}|t-s|^{\alpha}\leq \| Y \|_{\alpha,\eta}\eps^{\eta-\alpha}|t-s|^{\alpha}.\]
By taking the sup on $s$ and $t$ we obtain 
\[C := \sup_{\eps \in (0,T]}\| Y \|_{\alpha; [\eps,T]}\eps^{\alpha- \eta}\leq \| Y \|_{\alpha,\eta}\,.\]
Conversely, for any $\eps \in (0,T]$,
\begin{equation}\label{eq:best_constant}
\| Y \|_{\alpha; [\eps,T]}|t-\eps|^{\alpha}\leq C\eps^{\eta- \alpha}|t-\eps|^{\alpha}\,.
\end{equation}
Since $Y$ is $\alpha$-H\"older on $[\eps,T] $ we have trivially
\[
|Y_t-Y_{\eps}|\leq \| Y \|_{\alpha; [\eps,T]}|t-\eps|^{\alpha}\leq C\eps^{\eta- \alpha}|t-\eps|^{\alpha}\,.
\]
By taking the sup on $\eps$ and $t$, one has $\| Y \|_{\alpha,\eta} \leq C$ and the proof is finished.
\end{proof}
Since we are considering the sup over a fixed set $(0,T]$, we can equivalently modify the underlying set of values $s$ such that the quantities $|t-s|$ and $s$ are comparable. {\color{black}The following  equivalence of seminorms will allow us to compare singular H\"older spaces with  singular modelled distributions in section \ref{sec_reg_struct}}.
\begin{lemma}\label{lem:equality}
{Let $\alpha \in (0,1)$, $\delta>0$ and $\eta\leq  \delta$. For any function $Y\colon (0,T]\to\bR$ there exists a constant $C>0$ depending on \color{black} $\alpha,\eta,\delta$ such that one has}
\begin{equation}\label{eq:equivalence}
 \sup_{\substack{0<s<t\leq T\\ |t-s|\leq s}} \frac{|Y_t - Y_s|}{s^{\eta-\delta}|t-s|^{\alpha}}\leq \sup_{\substack{0<s<t\leq T}} \frac{|Y_t - Y_s|}{s^{\eta-\delta}|t-s|^{\alpha}}\leq C\sup_{\substack{0<s<t\leq T\\ |t-s|\leq s}} \frac{|Y_t - Y_s|}{s^{\eta-\delta}|t-s|^{\alpha}}\,.
\end{equation}
\end{lemma}
\begin{proof}
We use the shorthand notation $[Y]$ to denote the {\color{black} seminorm where the $\sup$ contains the condition $|t-s|\leq s$}. We fix $0<s<t$ and using the shorthand notation  $N=\lfloor\log_2(t/s)\rfloor$, we consider the sequence of values
\begin{equation}
s_n=\begin{cases}t & n=N+1  \\2^ns & n\leq N\,.\end{cases} 
\end{equation}
Since these values satisfy $|s_{n+1}- s_n|\leq s_n$ for every $n\leq N$, we can iterate the triangle inequality obtaining
\begin{equation}\label{eq:triangle_ineq}
\begin{split}
|Y_t-Y_s|&\leq \sum_{n=0}^{N}|Y_{s_{n+1}}- Y_{s_n}|\leq [Y] \sum_{n=0}^{N}|s_{n+1}- s_n|^{\alpha}s_{n}^{\eta-\delta}\\&= [Y](2^Ns)^{\eta -\delta}\left((t- 2^Ns)^{\alpha}+ 2^{-N(\eta-\delta)}\frac{1-2^{N (\alpha+\eta-\delta)}}{1- 2^{(\alpha+\eta-\delta)}}s^{\alpha}\right)\,.
\end{split}
\end{equation}
When {\color{black}$\alpha+\eta -  \delta>0$} we can easily estimate the second line of \eqref{eq:triangle_ineq} obtaining
\[|Y_t-Y_s|\leq [Y]s^{\eta -\delta} \left((t- 2^Ns)^{\alpha}+ \frac{2^{N\alpha}-1}{2^{\alpha+ \eta- \delta}-1} s^{\alpha}\right)\,.\]
Using the definition of $N$ and the H\"older behaviour of the function $t\to t^{\alpha}$ there exists a constant $C'>0$ depending on $\alpha,\eta,\delta$ such that
\[\frac{2^{N\alpha}-1}{2^{\alpha+ \eta- \delta}-1}\leq \frac{1}{2^{\alpha+ \eta- \delta}-1}\frac{(t^{\alpha}- s^{\alpha})}{s^{\alpha}}\leq {\color{black}C'}\frac{(t- s)^{\alpha}}{s^{\alpha}}\,.\]
Therefore we obtain
\[|Y_t-Y_s|\leq (1+C')[Y]s^{\eta- \delta}(t-s)^{\alpha}\,.\]
Dividing this expression by $s^{\eta- \delta}(t-s)^{\alpha}$ and taking the sup, we conclude. In case $\alpha+\eta - \delta=0$, {\color{black} we can repeat all the calculation in \eqref{eq:triangle_ineq}, obtaining the inequality
\begin{equation}\label{eq:triangle_ineq_2}
\begin{split}
|Y_t-Y_s|\leq [Y] \left((t- 2^Ns)^{\alpha}s^{-\alpha}+N\right)\,.
\end{split}
\end{equation}
Using the elementary estimate
\[N \leq \frac{2^{\alpha N}-1}{\alpha \ln(2)}\]
and the definition of $N$, there exist a constant $c>0$ depending on $\eta$, $\delta$ and $\alpha$, such that 
\[|Y_t-Y_s|\leq [Y](1+c)(t-s)^{\alpha}s^{-\alpha}\,.\]
In case $\alpha+\eta- \delta< 0$ we bound  the second line of \eqref{eq:triangle_ineq} with
\[|Y_t-Y_s|\leq [Y] s^{\eta -\delta}\left((t- 2^Ns)^{\alpha}+ \frac{N\wedge 1}{1-2^{\alpha+ \eta- \delta}} s^{\alpha}\right)\,.\]
Using again the definition of $N$, we see that when $N\geq 1$ then $ s< |t-s|$ and there exists a constant $c'>0$ such that
\[|Y_t-Y_s|\leq [Y](1+c')(t-s)^{\alpha}s^{\eta -\delta}\,,\]
thereby obtaining the result.}
\end{proof}

An interesting characterisation of the functions $\CC^{\alpha, \eta}((0,T])$ is provided when $\eta>0$.
\begin{proposition}\label{prop_charach}
Let $0 < \alpha <  1$ and $0<\eta \leq \alpha$. A continuous path $Y$ belongs to $  \CC^{\alpha,\eta}((0,T])$ if and only if  the  re-parametrisation 
$t\to Y(t^{\alpha/\eta})$ belongs to $\CC^{\alpha}([0,T])$.
\end{proposition}
\begin{proof}
We suppose that $Y\in  \CC^{\alpha,\eta}((0,T])$ and we adopt the notation $\beta=\alpha/\eta$. By hypothesis, {\color{black} there exist two  constants $M, M'>0$ such that for any $t>s>0$} one has 
\[\begin{split}
|Y(t^{\beta})-Y(s^{\beta})|&\leq M s^{\beta(\eta- \alpha)}|t^{\beta}- s^{\beta}|^{\alpha}= M \beta s^{\beta(\eta- \alpha)}\left| \int_s^tr^{\beta-1} dr\right|^{\alpha}\\&\leq  M' s^{\beta(\eta- \alpha)}s^{(\beta-1)\alpha}|t-s|^{\alpha}= M' |t-s|^{\alpha}\,.
\end{split}\]
We conclude that  the re-parametrisation can be uniquely extended to a $\alpha$-H\"older function on $[0,T]$. On the other hand, if the function  $t\to Y(t^{\beta})$ is $\alpha$-H\"older there exists a constant $C>0$ such that for any $0<s<t\leq T$ one has 
\[\begin{split}
|Y(t)-Y(s)|&\leq C|t^{1/\beta}- s^{1/\beta}|^{\alpha}= \frac{C}{\beta^{\alpha}}\left |\int_s^tr^{1/\beta-1}dr\right|^{\alpha}\leq \frac{C}{\beta^{\alpha}} s^{\alpha/\beta-\alpha} | t-s|^{\alpha}=\frac{C}{\beta^{\alpha}} s^{\eta-\alpha} | t-s|^{\alpha}\,.
\end{split}
\]
Diving both sides by $s^{\eta-\alpha} | t-s|^{\alpha}$, we conclude characterisation.
\end{proof}

\subsection{Improper Young integration} 
For any couple  $Y\in \CC^{\alpha_1}((0,T])$ and  $X\in \CC^{\alpha_2}((0,T])$ satisfying $\alpha_1+ \alpha_2>1$, we can define for any  $t>s>0$ the Young integral 
\[\int^t_{s}Y_rdX_r\,.\]
Following e.g. \cite[Chap.4]{FH2020} there exists a constant $C>0$ such that
\begin{equation}\label{eq:int_estimate}
\left\vert \int_s^t Y_r d X_r - Y_s(X_t-X_s) \right\vert \leq   C\| X \|_{\alpha_2; [s,t]} \| Y \|_{\alpha_{1}; [s,t]} |t -s|^{\alpha_1+ \alpha_2}\,,
\end{equation}
for every $T>t>s>0$. The following theorem establishes a sufficient criterion to extend this integral to $(0,T]$  by showing the existence the corresponding improper Young integral.
\begin{theorem}\label{thm:improper_young}
Let $Y\in \CC^{\alpha_1, \eta_1}((0,T])$ and  $X\in \CC^{\alpha_2, \eta_2}((0,T])$ such that $\alpha_1+ \alpha_2>1$. Provided the conditions $\eta_2>0$, $\eta_1+\eta_2>0$ and $\eta_1\neq 0$, one has the convergence
\begin{equation}\label{eq:improper_young}
\int^t_{0^+}Y_rdX_r:=\lim_{\eps\to 0^+}\int^t_{\eps}Y_rdX_r\,,
\end{equation}
for any $0<t\leq T$. We call the left-hand side of \eqref{eq:improper_young} the improper Young integral. As function of $t$ it belongs to $\CC^{\alpha_2, \eta}((0,T])$, where $\eta:=\eta_1\wedge 0+\eta_2>0 $. Moreover there exists a constant $M>0$ depending on $T$ and the previous parameters such that
\begin{equation}\label{thm:est_int}
\left\| \int^\cdot_{0^+}Y_rdX_r \right\|_{\alpha_2, \eta}\leq M(|Y_T|+ \|Y\|_{\alpha_1, \eta_1})\|X\|_{\alpha_2, \eta_2}\,.
\end{equation}
\end{theorem}
\begin{remark}
Before giving the proof, we stress the optimality of the conditions via the example $Y:t \mapsto t^{\eta_1}$ and $X: t \mapsto t^{\eta_2}$. In this case, $\int_s^T Y dX =  \text{(const)} T^{\eta_1 + \eta_2}-s^{\eta_1 + \eta_2}$ and the {\color{black} $\lim_{s\to 0^+}$} exists iff $\eta_1+\eta_2>0$. On the other hand, taking $Y \equiv 1$ shows that $X$ must extend continuously from $(0,T]$ to $[0,T]$, which is captured by the condition $\eta_2 > 0$. { \color{black} We also remark that the condition $\eta_1+\eta_2>0$ allows $\eta_1$ to be strictly negative, which is the non-trivial part of this result. Indeed if $\eta_1>0$ we can
use Proposition \ref{prop:char_alphaeta_Holder}  to conclude that $X, Y$ have appropriate $p$-variation regularity and the result will follow from classical Young integration in the $p$-variation setting.}
\end{remark}

\begin{proof}
It is sufficient to prove the convergence \eqref{eq:improper_young} when $t=T=1$. In order to show the convergence to some value, we will firstly prove that $I_n := \int_{2^{-n}}^1 Y dX$ is a Cauchy sequence. By means of the notation $Y^n=Y_{2^{-n}}$, we apply the inequality \eqref{eq:int_estimate} with $s=2^{-n}, t = 2^{-n+1}$ obtaining
\begin{equation}\label{eq:thm_est_int1}
\begin{split}
|I_n- I_{n-1}|\leq |Y^{n}| \| X \|_{\alpha_2; [2^{-n},2^{-n+1}]}2^{-n\alpha_2} +  C \| X \|_{\alpha_2; [2^{-n},2^{-n+1}]} \| Y \|_{\alpha_{1}; [2^{-n},2^{-n+1}]} 2^{-n(\alpha_1+ \alpha_2)}\,.
\end{split}
\end{equation}
{\color{black}Thanks} to Proposition \ref{prop:char_alphaeta_Holder}, we can derive some bounds of each quantity in the right-hand side above. Iterating the triangle inequality, there exists a constant $C'>0$ such that  for every $n \geq 1$
\[
\begin{split}
|Y^{n}|&\leq \sum_{m=1}^{n}|Y^{m}- Y^{m-1}|+ |Y^0|\leq \| Y\|_{\alpha_1, \eta_1}\sum_{m=1}^{n} 2^{-m \eta_1} + |Y_1|\leq C'(2^{-n(\eta_1\wedge 0)}\| Y\|_{\alpha_1, \eta_1}) +|Y_1|\, {\color{black} .}
\end{split}
\]
Plugging this estimate in \eqref{eq:thm_est_int1}, there exists a constant $c>0$ such that
\begin{equation}\label{proof:improper_young5}
\begin{split}
|I_n- I_{n-1}|&\leq  \|X\|_{\alpha_2, \eta_2}\bigg(C\| Y\|_{\alpha_1, \eta_1} (C'2^{-n(\eta_1\wedge 0+\eta_2)}+ 2^{-n(\eta_1+ \eta_2)}) +|Y_1|2^{-n\eta_2}\bigg)\\&\leq c\|X\|_{\alpha_2, \eta_2} (\|Y\|_{\alpha_1, \eta_1} +|Y_1|)2^{-n\eta}\,.
\end{split}
\end{equation}
Then the conditions in the statement guarantee that $I_n$ is a Cauchy sequence converging to some limit $I$ such that for every $n\geq 0$
\begin{equation}\label{proof:improper_young3}
| I - \int^1_{2^{-n}}Y_rdX_r|\leq c\|X\|_{\alpha_2, \eta_2} (\|Y\|_{\alpha_1, \eta_1} +|Y_1|)2^{-n\eta}\,,\end{equation} Let us prove the convergence of the whole sequence in the right-hand side of \eqref{eq:improper_young}. For any fixed $\eps\in (0,1]$, we take $N$ such that $2^{-N}< \eps\leq 2^{-N+1} $  and applying the estimate \eqref{eq:int_estimate} with $s=2^{-N}$ and $t=\eps$, one has 
\[ \begin{split}
\left\vert \int^{\eps}_{2^{-N}}Y_rdX_r\right\vert\leq& |Y^{N}| \| X \|_{\alpha_2; [2^{-N},\eps]}| \eps -2^{-N}|^{\alpha_2} +  C\| X \|_{\alpha_2; [2^{-N},\eps]} \| Y \|_{\alpha_{1}; [2^{-N},\eps]} | \eps -2^{-N}|^{\alpha_1+\alpha_2}\\\leq & |Y^{N}| \| X \|_{\alpha_2; [2^{-N},2^{-N+1}]}2^{-N\alpha_2} +  C\| X \|_{\alpha_2; [2^{-N},2^{-N+1}]} \| Y \|_{\alpha_{1}; [2^{-N},2^{-N+1}]} 2^{-N(\alpha_1+ \alpha_2)}\,.
\end{split}\]
Then, we can repeat exactly the calculations written above to obtain the estimate
\begin{equation}\label{proof:improper_young4}
\left\vert \int^{\eps}_{2^{-N}}Y_rdX_r\right\vert\leq c \|X\|_{\alpha_2, \eta_2} (\|Y\|_{\alpha_1, \eta_1} +|Y_1|)2^{-N\eta}\leq c \|X\|_{\alpha_2, \eta_2} (\|Y\|_{\alpha_1, \eta_1} +|Y_1|) \eps^{\eta}.
\end{equation} 
Combining the two inequalities \eqref{proof:improper_young3} and \eqref{proof:improper_young4}, we obtain the convergence \eqref{eq:improper_young}. Finally let us show that $\int_{0^{+}}^{\cdot}Y_r dX_r$ belongs to $\CC^{\alpha_2, \eta}$. The case $\eta= \alpha_2$ follows trivially from \eqref{proof:improper_young3}. When $\eta<\alpha_2$ we use the characterisation \eqref{prop:ident_norms} obtaining 
\[\left\Vert\int_{0^+}^{\cdot} Y_r d X_r\right\Vert_{\alpha_2, \eta}= \sup_{\eps>0}\left\Vert\int_{0^+}^{\cdot} Y_r d X_r\right\Vert_{\alpha_{2}; [\eps,1]}\eps^{\alpha_2- \eta}\leq \sup_{n\geq 1}\left\Vert\int_{0^+}^{\cdot} Y_r d X_r\right\Vert_{\alpha_{2}; [2^{-n},1]}2^{-n(\alpha_2- \eta)}\,.\]
By means of the subadditivity of the H\"older semi-norm and the inequality \eqref{proof:improper_young5}, we obtain for any $n\geq 1$
\[
\begin{split}
\left\Vert\int_{0^+}^{\cdot} Y_r d X_r\right\Vert_{\alpha_{2}; [2^{-n},1]}&\leq \sum_{m=1}^n\left\Vert\int_s^t Y_r d X_r\right\Vert_{\alpha_{2}; [2^{-m},2^{-m+1}]}\leq \sum_{m=1}^n|I_m- I_{m-1}|2^{m\alpha_{2}}\\&\leq c\|X\|_{\alpha_2, \eta_2} (\|Y\|_{\alpha_1, \eta_1} +|Y_1|)\sum_{m=1}^n2^{-m(\eta-\alpha_{2})}\\&\leq  c'\|X\|_{\alpha_2, \eta_2} (\|Y\|_{\alpha_1, \eta_1} +|Y_1|)2^{-n(\eta-\alpha_{2})}\,,
\end{split}
\]
for some constant $c'>0$. Therefore $\left\Vert\int_{0^+}^{\cdot} Y_r d X_r\right\Vert_{\alpha_2, \eta}$ is finite and one has the inequality \eqref{thm:est_int}. The general case when $T$ and $t$ are arbitrary follows from the scaling properties of H\"older semi-norms.
\end{proof}

\begin{remark}
As a special case of interest, for $X\in \CC^{\alpha}[0,T]$, $Y\in \CC^{\alpha, \eta}(0,T]$ and  $\alpha>1/2$, we identify $\CC^{\alpha}[0,T]$ with $\CC^{\alpha, \alpha}(0,T] $ and Theorem \ref{thm:improper_young} yields the improper Young integral provided $\eta + \alpha > 0$. 
\end{remark}

\subsection{Singular Besov paths}
It is known that paths of fractional Besov regularity, $X \in W^{\delta ,q}([0,T])$ enjoy $\alpha$-H\"older regularity provided $\alpha = \delta - 1/q > 0$ and also finite $p$-variation with $p = 1/\delta$ (see e.g. \cite{FV06var}). Let us introduce a singular version of this space.
\begin{definition}\label{defn_singular_Besov}
Let $0 < \delta <  1$, $q\geq 1$ such that $\alpha=\delta - 1/q>0$ and $\eta \le \alpha$. We say that a continuous path $Y\colon (0,T]\to \R$ belongs to the \emph{singular Besov space} $ W^{\delta ,q, \eta}((0,T])$  if it satisfies
\begin{equation}\label{defn_besov_norm}
\left\Vert Y \right\Vert _{\delta,q,\eta}=\left( \int_{\{ 0 < s < t < T \}} \frac{\left\vert Y_{t}-Y_{s}\right\vert ^{q}}{ s^{q(\eta - \alpha)} \left\vert t-s\right\vert^{1+\delta q}}dsdt\right) ^{1/q} <\infty\;.
\end{equation}
\end{definition}
Out interest in singular Besov path spaces comes from the following embedding result.
\begin{proposition}\label{prop:char_alphaeta_Sobolev_Holder}
We have the continuous embedding
\begin{equation}\label{sing_besov_embedding}
 W^{\delta ,q, \eta}((0,T]) \subset C^{\alpha, \eta} ((0,T])\,.
\end{equation}
\end{proposition}
\begin{proof}
The proof is similar to Proposition \ref{prop:char_alphaeta_Holder}. Indeed, for every $\varepsilon \le s < t < T$ have, by the usual Besov
embedding, with $\alpha = \delta - 1 / q > 0$, and using $\eta \le \alpha$ in the estimation below, 
\begin{eqnarray*}
  | Y_t - Y_s | & \le & 
  \left( \int_{\{ \varepsilon < s < t < T \}} \frac{\left\vert Y_{t}-Y_{s}\right\vert ^{q}}{\left\vert t-s\right\vert^{1+\delta q}}dsdt\right) ^{1/q}
   | t - s |^{\alpha} \\
   & \le &  \varepsilon^{\eta -\alpha}  \left( \int_{\{ \varepsilon < s < t < T \}} \frac{\left\vert Y_{t}-Y_{s}\right\vert ^{q}}{ s^{q(\eta - \alpha)} \left\vert t-s\right\vert^{1+\delta q}}dsdt\right) ^{1/q}
   | t - s |^{\alpha} \\
  & \le & \varepsilon^{\eta -\alpha} \| Y \|_{\delta, q, \eta} | t - s
  |^{\alpha}
\end{eqnarray*}
In particular, we see $\| Y \|_{\alpha ; [\varepsilon, T]} \le
\varepsilon^{\eta - \alpha} \| Y \|_{\delta, q, \eta}$ and hence $\| Y
\|_{\alpha, \eta} \le \| Y \|_{\delta, q, \eta}$, using Proposition \ref{prop:char_alphaeta_Holder}.
\end{proof}
\begin{remark}
We note a notational clash, (classical) Besov $\| Y \|_{\delta,q}$ vs. singular H\"older $\| Y \|_{\alpha, \eta}$, but will not try to resolve this for the simple reason that such Besov spaces will 
not play a central role in this paper. {\color{black} Definition} \ref{defn_singular_Besov} is done for real-valued paths and we can trivially extend it when $Y$ takes value in the complex upper half-plane $\mathbb{H}$ (see Section \ref{SLE_sec}).
\end{remark}

\section{Singular rough path spaces} \label{sec:singRPS}

{\color{black} We extend the previous on Young integration in Theorem \ref{thm:improper_young} to usual rough integration with respect to a specific two-level family of rough paths.} In what follows all the definitions are given with respect to two fixed values $\alpha,\beta\in (0,1)$ such that $\alpha+2\beta > 1$.
\begin{definition}\label{defn_rough_path}
An inhomogeneous rough path is a triple $\bX=(X,\hat X, \XX)$ where $X\in \CC^{\alpha}([0,T])$, $\hat{X}\in \CC^{\beta}([0,T])$ and $\XX\in \CC^{\alpha+ \beta}([0,T]^2)$ satisfies the Chen relation
\[
\XX_{s,t}= \XX_{s,u}+ \XX_{u,t}+ (\hat{X}_u-\hat{X}_s)(X_t-X_u)\,.\]
We denote the set of inhomogeneous rough paths by $\Cr([0,T])$. 
\end{definition}

For any given inhomogeneous rough path we define the corresponding space of paths which can be integrated against it.

\begin{definition}
Let $\bX\in \Cr([0,T])$ and $\gamma\in (0,1)$ such that $\alpha+\beta+ \gamma > 1$. An inhomogeneous controlled rough path is given by a couple $(Y, Y')$ where $Y\in \CC^{\beta}([0,T])$, $Y' \in \CC^{\gamma}([0,T])$ and the remainder $R$, given by 
\begin{equation}\label{remainder}
Y_t = Y_s+ Y'_s (\hat{X}_t-\hat{X}_s) + R^Y_{s,t}\,,
\end{equation}
satisfies $R^Y \in \CC^{\gamma + \beta}([0,T]^2)$.  We denote the set of inhomogeneous controlled rough paths by $ \D_{\hat{X}}^{\gamma+\beta}([0,T])$. {\color{black} If the remainder $R^Y$ is taken with respect to $X$ and $R^Y \in \CC^{\gamma + \alpha}([0,T]^2)$, we also say that $(Y,Y')$ is an inhomogeneous controlled rough path and we denote the associated space by $\D_{X}^{\gamma+\alpha}([0,T])$}.
\end{definition}
We introduce their corresponding  singular version.

\begin{definition}\label{def_sing_inhomogeneous}
Let $\eta\leq \alpha+ \beta$. We say that  $\bX=(X,\hat{X},\XX)$, defined on $(0,T]$ is a singular inhomogeneous rough path, in symbols $\bX \in \Cr_{\eta}((0,T])$,  if it satisfies
\[
\|\bX\|_{\eta}:=\sup_{0 < s < t \le T} \frac{|X_t - X_s|}{s^{\eta -(\alpha+ \beta)}|t-s|^{\alpha}}+\sup_{0 < s < t \le T} \frac{|\hat{X}_t - \hat{X}_s|}{s^{\eta -(\alpha+ \beta)}|t-s|^{\beta}} +   \sup_{0 < s < t \le T} \frac{|\XX_{s,t}|}{s^{\eta -(\alpha+ \beta)}|t-s|^{\alpha+ \beta}} < \infty
\]
\end{definition}
\begin{definition}\label{def_sing_inhomogeneous_contr}
Let $\gamma\in (0,1)$ such that $\alpha+\beta+ \gamma > 1$ {\color{black} and} $\eta\leq \gamma+ \beta$. We call $(Y,Y')$ a singular  controlled rough path, in symbols $(Y, Y') \in \D_{\hat{X}}^{\gamma+ \beta, \eta}((0,T])$  if it satisfies
\[
\| Y,Y' \|_{\gamma+\beta,\eta} := \sup_{0 < s < t \le T} \frac{|Y'_t - Y'_s|}{s^{\eta -(\gamma+ \beta)}|t-s|^{\gamma}} +   \sup_{0 < s < t \le T} \frac{|R^Y_{s,t}|}{s^{\eta -(\gamma+ \beta)}|t-s|^{\gamma+ \beta}} < \infty\,.
\]
{\color{black}If $R^Y$ is taken with respect to $X$ and we replace $\beta$ with $\alpha$, we denote the related space by $\D_{X}^{\gamma+ \alpha, \eta}((0,T])$.}
\end{definition}
Using the intrinsic norms of inhomogeneous rough paths and controlled rough paths\footnote{They are defined as $\| Y,Y' \|_{\gamma+ \beta,\eta} $ and $\|\bX\|_{\eta}$ without the weight $s^{\eta- (\alpha+ \beta)}$ or $s^{\eta-( \gamma+\beta)}$, see \cite{FH2020}}, one can obtain the same type of equivalences given in Proposition \ref{prop:char_alphaeta_Holder}.
\begin{proposition}
Let {\color{black} $\bX\in \Cr_{\eta}((0,T])$}  and $Y\in \D_{\hat{X}}^{\gamma+ \beta, \eta}((0,T])$. Considering the quantities $\| \bX\|_{[\eps,T]}$ and  $\| Y,Y' \|_{\gamma+ \beta; [\eps, T]} $  one has the same results of Proposition \ref{prop:char_alphaeta_Holder} obtaining 
\begin{equation}\label{eq:equivalence-rough}
 \|Y, Y'\|_{\gamma+ \beta, \eta}=\sup_{0<\eps\leq T}\| Y,Y' \|_{\gamma+ \beta; [\eps, T]}\eps^{\gamma+ \beta- \eta}\,,  \quad\| \bX\|_{\eta}= \sup_{0<\eps\leq T} \| \bX\|_{ [\eps,T]}\eps^{\alpha+\beta- \eta}\,.
\end{equation}
\end{proposition}
\begin{proof}
The proof follows in same way as Proposition \ref{prop:char_alphaeta_Holder}.
\end{proof}
From \eqref{eq:equivalence-rough} one has immediately  the identity $\Cr_{\alpha+ \beta}((0,T])= \Cr([0,T])$ moreover one has the continuous embedding
\begin{equation}\label{eq:trivial_properties}
\D_{\hat{X}}^{\gamma'+ \beta, \eta}((0,T])\subset\D_{\hat{X}}^{\gamma+ \beta, \eta}((0,T])\,, \quad\text{if $\gamma'\geq \gamma$}\,.
\end{equation}
Similarly to the Young integral, for any  $(Y, Y')\in \cD^{\gamma+ \beta}_{\hat{X}}((0,T])$ and  $\bX\in \Cr((0,T])$  we can define (see \cite[Chap.4]{FH2020}) for any  $t>s>0$ the rough integral
\[\int_s^t Y_r d \bX_r\,.\] 
This function belongs to $\CC^{\alpha}((0,T])$ and it satisfies the inequality
\begin{equation}\label{eq:sewing_rough}
\left\vert \int_s^t Y_r d \bX_r- Y_s(X_t-X_s) -Y'_s\XX_{s,t}\right\vert \leq   C\| \bX\|_{ [s,t]}\|Y, Y'\|_{\gamma+ \beta; [s,t]}|t -s|^{\gamma+ \beta+ \alpha}\,,
\end{equation}
for some fixed constant $C>0$. We describe some sufficient conditions to show the existence of an improper rough integral in this case.
{\color{black}
\begin{theorem}\label{thm:improper_rough}
Let $Y\in\D_{\hat{X}}^{\gamma+\beta, \eta_1}((0,T])$ and  $\bX\in \Cr_{\eta_2}((0,T])$ such that $\alpha+\beta + \gamma > 1$. Provided the conditions $\eta_2>\beta \vee \alpha$, $2(\eta_2- \beta) +\eta_1- \alpha>0$ and $\eta_1\neq 0, \beta$ one has the convergence for any $0<t\leq T$
\begin{equation}\label{eq:convergence_integral}
Z_t:=\int^t_{0^+}Y_rd\bX_r:=\lim_{\eps\to 0^+}\int^t_{\eps}Y_rd\bX_r\,.
\end{equation}
We call the left-hand side of {\color{black}\eqref{eq:convergence_integral} the improper rough integral.} The path $Z$ belongs to $\CC^{\alpha, \eta}((0,T])$ and the couple $(Z, Y)$ belongs to $\cD^{ \beta+\alpha, \eta}_X((0,T])$, where $\eta:=\eta_2-\beta+(\eta_2- \alpha+(\eta_1 -\beta)\wedge 0 )\wedge 0$. Moreover, there exists a constant $M>0$ depending on $T$ and the previous parameters such that
\begin{equation}\label{eq:thm_1}
\| Z\|_{\alpha, \eta}\leq M(\|\bX\|_{\eta_2}\vee \|\bX\|_{\eta_2}^2)(|Y_T|+ |Y'_T|+ \| Y,Y' \|_{\gamma+ \beta,\eta})\,,
\end{equation}
\begin{equation}\label{eq:thm_2}
\left\|  Z, Y\right\|_{\alpha +\beta, \eta}\leq M(\|\bX\|_{\eta_2}\vee \|\bX\|_{\eta_2}^2)( |Y'_T|+ \| Y,Y' \|_{\gamma+ \beta,\eta})\,.
\end{equation}
\end{theorem}
\begin{proof}
The proof is {\color{black} obtained} using the same strategy as Theorem \ref{thm:improper_young}. We fix again $t=T=1$ and we introduce the notation $Y^n=Y_{2^{-n}}$, $\dot{Y}^n= Y'_{2^{-n}}$ and $Z_n := \int_{2^{-n}}^1 Y d\bX$. Using  the inequality \eqref{eq:sewing_rough} with $s=2^{-n}, t = 2^{-n+1}$, we obtain
\begin{equation}\label{eq:proof_rough_int}
\begin{split}
|Z_n- Z_{n-1}| \leq  &  |Y^n|  \| X \|_{\alpha; [2^{-n},2^{-n+1}]}2^{-\alpha n}+ |\dot{Y}^n|  \| \XX \|_{\alpha + \beta; [2^{-n},2^{-n+1}]}2^{-(\alpha+ \beta) n}\\&+ C\|\bX\|_{ [2^{-n},2^{-n+1}]}\|Y, Y'\|_{\gamma+ \beta; [2^{-n},2^{-n+1}]}2^{-(\alpha+ \beta+ \gamma) n}\,.
\end{split}
\end{equation}
Iterating the triangle inequality, there exists a constant $C'>0$ such that for every $n \geq 1$  one has
\begin{equation}\label{eq:estimate_deriv}
\begin{split}
|\dot{Y}^n|&\leq \sum_{m=1}^{n}|\dot{Y}^m- \dot{Y}^{m-1}|+ |\dot{Y}^0|\\&\leq  \sum_{m=1}^{n} 2^{-m (\eta_1-\beta)}\|Y, Y'\|_{\gamma + \beta,\eta_1} + |Y'_1|\\&\leq C'\|Y, Y'\|_{\gamma+ \beta,\eta_1} 2^{-n((\eta_1 -\beta)\wedge 0)}+ |Y'_1|\,,
\end{split}
\end{equation}
which implies
\begin{equation}\label{eq:estimate_derY}
\begin{split}
|\dot{Y}^n|\| \XX \|_{\alpha + \beta; [2^{-n},2^{-n+1}]}2^{-(\alpha + \beta )n}&\leq \|\bX\|_{\eta_2}\bigg(C'\|Y, Y'\|_{\gamma+ \beta,\eta_1} 2^{-n(\eta_2+(\eta_1 -\beta)\wedge 0))}+|Y'_1|2^{-n\eta_2}\bigg)\,.
\end{split}
\end{equation}
Passing to the first term in the right-hand side of \eqref{eq:proof_rough_int}, we can apply again \eqref{eq:estimate_deriv} to obtain for any $n\geq 1$
\begin{equation}\label{eq:estimate_Y}
\begin{split}
|Y^{n}|&\leq|Y^{n-1}|+|R^{Y}_{2^{-n},2^{-n+1}}|+ |\dot{Y}^{n}||( \hat{X}_{2^{-n}}- \hat{X}_{2^{-n+1}})| 
\\&\leq |Y^{n-1}|+   2^{-n\eta_1}\|Y, Y'\|_{\gamma+ \beta,\eta_1}\\&+ \|\bX\|_{\eta_2}( C'\|Y, Y'\|_{\gamma+ \beta,\eta_1} 2^{-n(\eta_2-\alpha+(\eta_1 -\beta)\wedge 0 )}+ |Y'_1|2^{-n(\eta_2-\alpha)})\,.
\end{split}
\end{equation}
Iterating recursively this estimate, we obtain that there exists a constant $C''>0$ such that
\[
|Y^{n}|\leq C''\|Y, Y'\|_{\gamma+ \beta,\eta_1}\bigg( 2^{-n(\eta_1\wedge 0)}+ \|\bX\|_{ \eta_2}(2^{-n((\eta_2- \alpha+(\eta_1 -\beta)\wedge 0 )\wedge 0))}+|Y'_1|2^{-n((\eta_2 -\alpha)\wedge 0) })\bigg) +|Y_1|{\color{black}\,.}
\]
Therefore we obtain 
\[
\begin{split}
|Y^{n}|\| X \|_{\alpha; [2^{-n},2^{-n+1}]}2^{-\alpha n}\leq \|\bX\|_{ \eta_2}\bigg(&C''\|Y, Y'\|_{\gamma+ \beta,\eta_1} \big(2^{-n(\eta_2- \beta+\eta_1\wedge 0)}+  \|\bX\|_{\eta_2}|Y'_1|2^{-n( \eta_2 - \beta
+(\eta_2 -\alpha)\wedge 0) } \\&+\|\bX\|_{\eta_2}2^{-n(\eta_2-\beta+(\eta_2- \alpha+(\eta_1 -\beta)\wedge 0 )\wedge 0))}  \big)+|Y_1|2^{-n(\eta_2- \beta)}\bigg)\,.
\end{split}
\]
Coming back to the initial estimate \eqref{eq:proof_rough_int}, there exist two constants $c, c'>0$ such that
\begin{equation}\label{eq:estimate_integral}
\begin{split}
|Z_n- Z_{n-1}|&\leq c( \|\bX\|_{ \eta_2}\vee \|\bX\|_{ \eta_2}^2)\bigg(\|Y, Y'\|_{\gamma+ \beta,\eta_1}\big(2^{-n(\eta_2- \beta +\eta_1\wedge 0)}+2^{-n(\eta_2+ \eta_1- \beta)}\\&+2^{-n(\eta_2-\beta+(\eta_2- \alpha+(\eta_1 -\beta)\wedge 0 )\wedge 0))})+ |Y'_1|(2^{-n( \eta_2 - \beta
+(\eta_2 -\alpha)\wedge 0) }\big) +|Y_1|2^{-n(\eta_2- \beta)}\bigg)\\&\leq c'(\|\bX\|_{ \eta_2}\vee \|\bX\|_{ \eta_2}^2)(\|Y, Y'\|_{\gamma+ \beta,\eta_1}+|Y'_1| +|Y_1|)2^{-n\eta}\,.
\end{split}
\end{equation}
By hypothesis on the parameters $\eta_1$ and $\eta_2$ one has $\eta>0$ and the sequence $I_n$ is a Cauchy sequence converging to some value $I$. Following the last part of Theorem \eqref{thm:improper_young}, we can repeat the same steps to deduce the convergence \eqref{eq:convergence_integral} and the first inequality \eqref{eq:thm_1}.  Let us show that the couple $(Z, Y)$ belongs to $\cD^{\alpha+ \beta, \eta}_{X}$. Firstly, we apply \eqref{eq:estimate_Y} to show that there exists a constant $c''>0$ such that
\begin{equation}\label{eq:proof_thm}
\begin{split}
\Vert Y \Vert_{\beta; [2^{-n}, 1]}&\leq \sum_{m=1}^n|Y^{m}- Y^{m-1}|2^{m\beta}\\&\leq  \sum_{m=1}^n|\dot{Y}^{m}|| \hat{X}_{2^{-m}}- \hat{X}_{2^{-m+1}}|2^{m\beta}+  \sum_{m=1}^n|  R^{Y}_{2^{-m},2^{-m+1}}|2^{m\beta}
\\&\leq c''(1\vee \|\bX\|_{ \eta_2})( \|Y, Y'\|_{\gamma+ \beta,\eta_1}+|Y'_1|)2^{-n((\eta_2-\alpha+ (\eta_1- \beta)\wedge 0))\wedge 0 +\alpha - (\beta +\alpha   ))}\,.
\end{split}
\end{equation}
and since $ (\eta_2-\alpha+ (\eta_1- \beta)\wedge 0))\wedge 0 +\alpha > \eta$ we conclude. On the other hand, we deduce from \eqref{eq:sewing_rough} and \eqref{eq:estimate_derY} that there exists a constant $c'''>0$ such that the function $R^Z_{s,t}=Z_t-Z_s-Y_s(X_t-X_s)$ satisfies for all $m\geq 1$
\begin{equation}\label{eq:estimate_I}
\begin{split}
|R^Z_{2^{-m+1},2^{-m}}|& \leq |\dot{Y}^m||\XX_{2^{-m},2^{-m+1}}|+ C\| \bX\|_{ [2^{-m},2^{-m+1}]}\|Y, Y'\|_{\gamma+ \beta; [2^{-m},2^{-m+1}]}2^{-(\alpha+ \beta+ \gamma) m}\\&\leq c'''\|\bX\|_{ \eta_2}(\|Y, Y'\|_{\gamma+ \beta,\eta_1} + |Y'_1|) 2^{-m(\eta_2+(\eta_1-\beta) \wedge 0)}\,.
\end{split}
\end{equation}
We introduce for every $n\geq 1$ the sequence
\[\gamma_{n}=\sup_{2^{-n}\leq s< t\leq 1}\frac{|R^Z_{s,t}|}{|t-s|^{\alpha+ \beta}}\,.\]
By splitting the above $\sup$ into two intervals and applying the two previous estimates \eqref{eq:proof_thm} and \eqref{eq:estimate_I}, we obtain for every $n\geq 1$
\[
\begin{split}
\gamma_n&\leq  \gamma_{n-1}+ |R^Z_{2^{-n},2^{-n+1}}| 2^{n(\alpha+ \beta)}+ \Vert Y \Vert_{\beta; [2^{-n}, 1]}\Vert X \Vert_{\alpha; [2^{-n+1}, 1]}\\&\leq  \gamma_{n-1}+ |R^Z_{2^{-n},2^{-n+1}}| 2^{n(\alpha+ \beta)}+ c'''\|\bX\|^2_{\eta_2}(\|Y, Y'\|_{\gamma+ \beta,\eta_1} + |Y'_1|) 2^{-n( \eta- (\alpha+ \beta))}\\&\leq \gamma_{n-1}+ (c'''\vee c'')(\|\bX\|_{ \eta_2}\vee \|\bX\|^2_{ \eta_2})(\|Y, Y'\|_{\gamma+ \beta,\eta_1} + |Y'_1|) 2^{-n(\eta- (\alpha+ \beta))}\,.
\end{split}
\]
Iterating recursively this inequality, we obtain a general estimate on the sequence $\gamma_n$ which implies trivially \eqref{eq:thm_2} and we conclude. The case of a generic $T>0$ is covered using the scaling properties of the norms $\|\bX\|_{ \eta_2}$ and $\|Y,Y'\|_{\gamma+ \beta, \eta_1}$.
\end{proof}

\begin{remark}
The present theorem is a generalisation of the classical rough integration as introduced in \cite{Gubinelli2004}. Indeed it is sufficient to set  $\alpha\in (1/3, 1/2]$, $\alpha=\beta $ and $\hat{X}= X$ to recover it. However, the proof of our result can be adapted to cover rough integration in a more general context when the underlying rough path is, for example, branched or geometric, see \cite{lyons1998,gub10}.

 Compared to Definition \ref{defn_rough_path}, these notions assume some additional algebraic conditions in their formulation but keep essentially the same H\"older-type structure. In addition, it is also possible to state an equivalent notion of rough integration for branched and geometric rough paths, which extends  \eqref{eq:sewing_rough}. Therefore, any new definitions, like singular-branched rough paths or singular-geometric rough paths, can be easily transferred to this new context. Moreover, any extension of Theorem \ref{thm:improper_rough} to these objects will involve only a careful check of more sophisticated dyadic powers without changing the proof strategy. Hence, we trust that the reader will be able to make suitable amendments on their own.
\end{remark}
}
%
\section{Application to Schramm-Loewner Evolution}\label{SLE_sec}
{\color{black}For any $\kappa>0$, we consider the family of conformal maps $(g_t)$ that are the maximal solution to Loewner's equation
\[ dg_t(z)= \frac{2}{g_t(z)- U_t}dt\,,\quad g_0(z)=z\in \mathbb{H}\]
where $U_t= \sqrt{\kappa} B_t$ with $(B_t)_{t\geq 0}$ a standard Brownian motion. Using the notations $f_t=(g_t)^{-1}$ and $\hat{f}_t(z):=f(z+U_t)$, the SLE trace $\gamma_{\kappa}$ (parametrised according to the half-plane capacity) can be defined a.s. for all $t\geq 0$ as the limit
\[\gamma_{\kappa}(t):=\lim_{u\to 0^+}\hat{f}_t(iu)\]
see \cite{rohde_schramm05,LSW11}.} In \cite{friz2017regularity} the H\"older regularity and $p$-variation of  $\gamma_\kappa$  were studied. In essence, there exists set $I(\kappa)\subset[0,\infty)$ depending $\kappa\in (0,\infty)$ and a constant $C>0$ not depending on $\kappa$ such that for any $r \in I(\kappa)$ one has the estimate
\begin{equation}\label{equ:SLE_moments}
\E[|\gamma_{\kappa}(t)-\gamma_{\kappa}(s)|^q]  \leq C s^{-\zeta/2} (t-s)^{(q + \zeta)/2} 
\end{equation}
where the parameters are functions $q=q(r)$ and $\zeta = \zeta (r)$ are given by\footnote{Strictly speaking, the exponents must be modified by arbitrarily small $\epsilon$, but for power counting arguments, this is good enough.}
\[ q\left( r\right) =\left( 1+\frac{\kappa }{4}\right) r-\frac{\kappa r^{2}}{8}\,, \qquad \zeta(r)=r- \frac{\kappa r^{2}}{8}\,.\]
We recall from \cite{friz2017regularity} that, at least for $\kappa \in (0,8)$ (our later interest concerns $\kappa \in (0,1)$ where boundary effects play a role) the set $I(\kappa)$ is defined introducing the auxiliary sets
\[\begin{split}
 I_{0}\left( \kappa \right) &:=\left\{ r\in \R%
:r<r_{c}(\kappa)\right\}\quad  \text{ with }r_{c}(\kappa)\equiv \frac{1}{2}+\frac{4}{\kappa }\,, \\
I_{1}\left( \kappa \right) &:=\left\{ r\in \R:q(r)>1\right\}\,, \quad
I_{2}\left( \kappa \right) :=\left\{ r\in \R:q(r)+\zeta(r)
>0\right\}\,.
\end{split}\]
And then we define $I (\kappa)  := I_{0}(\kappa)\cap I_{1}(\kappa)\cap I_{2} (\kappa)$.
Since we want to study the trace of SLE using the singular Besov spaces introduced in Definition \ref{defn_singular_Besov}, we introduce the set 
\[J_{2}(\kappa)=\left\{ r\in \R:\zeta \left(r\right) +q\left( r\right) >2\right\}\,\]
and we restrict the values of $r$ in $I(\kappa)\cap J_2(\kappa)$. Thanks to \cite[Lem 5.1]{friz2017regularity} one has the characterisation  
\begin{equation}\label{eq:charc_sets}
I(\kappa)\cap J_2(\kappa)= (1,r_c)
\end{equation}
and one has trivially that the set $(1/q, (\zeta+q)/2q))$ is non empty.
\begin{proposition}\label{Besov_inclusion}
Let $\kappa\in (0,1)$. For any $r\in I(\kappa)\cap J_2(\kappa)$ and $\delta\in (1/q, (\zeta+q)/2q))$ one has 
\[ \E \left\Vert \gamma_{\kappa} \right\Vert_{\delta ,q, \eta}^{q}<\infty\,, \]
for any parameter $\eta>0$ such that $\eta < (\delta -\zeta/(2q))$ and $\eta\leq (\delta-1/q)$.  
\end{proposition}
\begin{proof}
The result comes immediately from the definition of the singular Besov norm and the estimate \eqref{equ:SLE_moments}. Indeed we obtain
\begin{equation}\label{bound_expectation}
\E\left\Vert \gamma_{\kappa} \right\Vert_{\delta ,q, \eta}^{q}\leq C \int_{\{ 0 < s < t < T \}} s^{-q(\eta-\alpha)-\zeta/2}|t-s|^{(q+\zeta)/2-1-\delta q}ds dt
\end{equation} 
Thus we only need to check that the choice of $\eta$ and  $r$ in the statement implies that the integral on the right-hand side of \eqref{bound_expectation} is finite. First of all the condition on $\eta$ prevents the integral to be infinite at $s=0$ indeed one has trivially
\[-q(\eta-\alpha)-\zeta/2= -q\eta + \delta-1 -\frac{\zeta}{2} >- 1\,.\]
On the other hand to provide the integrability at the diagonal $t=s$ the condition $\delta<(\zeta+q)/2q)$ implies immediately
\[\frac{(q+\zeta)}{2}-1-\delta q>-1\,.\]
Thereby obtaining the result.
\end{proof}
Using the inclusion \eqref{sing_besov_embedding}, we can optimize over the range of admissible $r$ to include the SLE trace in a  singular H\"older space.
\begin{theorem}\label{thm_SLE}
Let $\kappa\in (0,1)$. With probability one, the SLE trace $\gamma_\kappa(t): 0 \le t \le 1$ takes values in the singular H\"older space $\CC^{\alpha, \eta}((0,1])$, where $\alpha < \alpha_\ast(\kappa)$  given by
\begin{equation}\label{opt_parameters}
\alpha_{\ast}\left( \kappa \right) =1-\frac{\kappa }{24+2\kappa -8\sqrt{\kappa +8}}\,
\end{equation}
and any parameter $\eta>0$ such that $\eta<1/2$ and $\eta \le \alpha$.
\end{theorem}
\begin{proof} 
Following the proof of \cite[Thm 6.1]{friz2017regularity}, one has that the quantity 
\[\frac{\zeta+q}{2q}-\frac{1}{q}=\frac{\zeta+q-2}{2q}\,, \]
is maximal on $I(\kappa)\cap J_2(\kappa)$ for the  explicit value $r_{\ast}=(-8+ 4\sqrt{8+\kappa})/\kappa$ thus yielding the desired exponent $\alpha_\ast(\kappa)$. The choice $\alpha<\alpha_{\ast}(\kappa)$ is done in accordance with Proposition \ref{Besov_inclusion}. In order to obtain the optimal $\eta$ we will optimize over the same set of parameters the function
\[\frac{\zeta+q-2}{2q}\wedge\left(\frac{\zeta+q}{2q}- \frac{\zeta}{2q}\right)=\frac{\zeta+q-2}{2q}\wedge \frac{1}{2}\,.\]
In case  $\kappa\in (0,1)$ an explicit calculation shows $\alpha_{\ast}(\kappa)> 1/2$. Therefore its maximal value is given by $1/2$.
\end{proof}
Applying  Proposition \ref{prop_charach} to the SLE trace, we obtain an interesting property of its trajectories, {\color{black}which was summarised in the property \eqref{equ:us2}.}
\begin{corollary}\label{cor_SLE}
For any $\kappa \in (0,1) $ and any $\alpha<\alpha_{\ast}(\kappa)$ the re-parametrisation $ t \mapsto \gamma_{\kappa}(t^{2})$ is a.s. an $\alpha$-H\"older paths in  $\CC^{\alpha}([0,1])$.
\end{corollary}
\begin{proof}
Proposition  \ref{prop_charach} implies  that for  $\alpha < \alpha_\ast(\kappa)$ given in \eqref{opt_parameters} and $\eta<1/2$ such that $\eta\leq \alpha$ the re-parametrisation of the SLE trace $\gamma_\kappa(t^{\alpha/\eta}): 0 \le t \le 1$ is a.s. an $\alpha$-H\"older path $\CC^{\alpha}([0,1])$.  By choosing $\eta$ such that the ratio $\beta=\alpha/\eta\leq 2$,  the reparametrisation $ t\to \gamma_{\kappa}(t^2)= \gamma_{\kappa}((t^{2/\beta})^{\beta})$ becomes the composition of a $\alpha$-H\"older function with a Lipschitz function. Thereby we obtain the thesis.
\end{proof}

\subsection{SLE trace as singular H\"older rough path}
We regard the SLE trace $\gamma_\kappa$ as a path in $\R^2$. The random tensor series 
$$
1+ \sum_{n=1}^\infty \int_{0 \le t_1 \le ... \le t_n \le 1}  (d\gamma_{\kappa})_{t_1} \otimes ... \otimes  (d\gamma_{\kappa})_{t_n} \quad \in T (( \R^2 )) 
$$
is called {\it SLE signature} and characterizes SLE trace, at least for $\kappa \le 4$, as was shown in \cite{horatioSLE}. The (componentwise) expectation, a deterministic element
in $T(( \R^2 ))$, is conjectured to determine the law of SLE and was partially computed in \cite{werness2012}. 
These iterated integral indeed make sense as ``variational'' Young integrals \cite{werness2012, friz2017regularity}. From a H\"older perspective,  they create an interesting scale of (singular) H\"older rough paths. 
\begin{theorem} Let $\kappa \in [0,8)$. The SLE trace $\gamma_\kappa$ can be lifted to an a.s. $\alpha$-H\"older geometric rough path for any $0<\alpha< \alpha_{\ast}(\kappa)$, when $\kappa  \in [1,8)$. Moreover,  in case $\kappa  \in [0,1)$ $\gamma_\kappa$ can be lifted to an a.s. singular $\alpha$-H\"older geometric rough path with singularity parameter $0<\eta<1/2$.
\end{theorem}
\begin{proof} By direct inspection, this is true for $\gamma_0(t) = 2 i \sqrt{t}$, so assume $\kappa \in (0,8)$. We first treat the case $\kappa \ge  1$ in which case we can ignore boundary effects and singular parameter. In view of the $p$-variation regularity established in \cite{werness2012, friz2017regularity}, namely $p = 1 + \kappa^-/8$, iterated integrals are well-defined, but this does not quantify any H\"older control. We thus need to use, the control function $\omega$ that emerges from the Besov regularity, i.e. the function
{\color{black}\[\omega(s,t):= C\Vert \gamma_{\kappa}\Vert^{1/\delta}_{\delta,q}(t-s)^{\alpha/\delta}
\,,\]
where $C$ is a specific numerical constant,  as established in \cite{FV06var} and $\Vert \cdot\Vert_{\delta,q} $ is the Besov norm. Applying this control $\omega$, we obtain immediately the correct H\"older regularity of $\gamma_{\kappa}$, see \cite[Theorem 4.2, Theorem 6.1]{friz2017regularity}. Moreover the same control $\omega$ provides the right  estimates to obtain the correct H\"older regularity of the iterated integrals, which identifies $\gamma_\kappa$ and its $\lfloor 1/ \alpha \rfloor$ iterated integrals as $\alpha$-H\"older rough path.}  (Since $\alpha<\alpha_{\ast}(\kappa) \leq 1/2$, it is inconsequential to regard it as $\alpha$-H\"older rough paths with singularity parameter $\eta$). It remains to deal with the case $0 < \kappa < 1$. In this case $\alpha_{\ast}(\kappa) > 1/2$ away from $0$, so that the trace is a level-$1$ (a.k.a. Young) path on $[\eps,1]$, iterated (H\"older) Young integration is valid. On $[0,1]$, because of \eqref{equ:LVL}, this argument fails with classical iterated Young integration, but we can reinstall it via improper Young integration of Theorem \ref{thm:improper_young} {\color{black} where $\alpha_1, \alpha_2>1/2$ and $\eta_1, \eta_2>0$}. Details are left to the interested reader.
\end{proof}  



\section{Connection with regularity structures}\label{sec_reg_struct}
Inhomogeneous rough paths can be equivalently described using the formalism of regularity structures (We refer to \cite{Hairer2014,FH2020} and we suppose the reader is familiar with its main concepts). To see this link we  introduce  the vector space
\[\mathcal{T}= \bigoplus_{l\in A}\mathcal{T}_{l}\,,\quad A= \{\alpha-1, \beta-1, \alpha+ \beta-1, 0, \alpha, \beta, \alpha+ \beta\}\,,\]
where  each subspace $T_l$ is one-dimensional. We represent the element of its canonical basis with the notations
\[\begin{gathered}
\mathcal{T}_{\alpha-1}= \langle {\color{blue}\Xi}\rangle\,, \quad \mathcal{T}_{\beta-1}= \langle {\color{blue}\hat{\Xi}}\rangle\,, \quad \mathcal{T}_{\alpha+ \beta-1}=\langle {\color{blue}\hat{X}\Xi}\rangle\,, \quad \mathcal{T}_{0}=\langle {\color{blue}\mathbf{1}}\rangle\,,\\\mathcal{T}_{\alpha}= \langle {\color{blue}X}\rangle\,, \quad \mathcal{T}_{\beta}= \langle {\color{blue}\hat{X}}\rangle\,, \quad \mathcal{T}_{\alpha+ \beta}=\langle {\color{blue}\XX}\rangle\,.
\end{gathered}\]
We introduce $G$, a group of linear automorphism $\Gamma_h\colon \mathcal{T}\to \mathcal{T}$, defined for any $h\in \mathbf{R}^3$, $h=(h_1,h_2,h_3)$ as
\[\begin{gathered}
\Gamma_h{\color{blue}\Xi}={\color{blue}\Xi} \,, \quad \Gamma_h{\color{blue}\hat{\Xi}}={\color{blue}\hat{\Xi}} \,, \quad \Gamma_h{\color{blue}\mathbf{1}}={\color{blue}\mathbf{1}}\,, \quad \Gamma_h{\color{blue}\hat{X}\Xi}={\color{blue}\hat{X}\Xi}+ h_2{\color{blue}\Xi}\,,\\ \Gamma_h{\color{blue}X}={\color{blue}X} + h_1{\color{blue}\mathbf{1}}\,, \quad \Gamma_h{\color{blue}\hat{X}}={\color{blue}\hat{X}} + h_2{\color{blue}\mathbf{1}} \,,  \quad \Gamma_h{\color{blue}\XX}={\color{blue}\XX}+ h_2{\color{blue}X}+ h_3{\color{blue}\mathbf{1}} \,.
\end{gathered}\]
{\color{black}
The couple $(\mathcal{T}, G)$ is a simple example of regularity structure. Following \cite[Lemma 13.20]{FH2020}, we can  rewrite every element $\bX\in  \Cr([0,T])$ as a specific \emph{model}  $\mathbf{\Pi}(\bX)$ over $(\mathcal{T}, G)$ restricted over $[0,T]$, see \cite[Definition 2.17]{Hairer2014} for the general definition. In few words, $\mathbf{\Pi}(\bX)$ is a couple of elements  $\mathbf{\Pi}(\bX)=(\Pi^{\bX},\Gamma^{\bX})$. The first one is a map $\Pi^{\bX}\colon [0,T]\to \mathcal{L}(\mathcal{T}, \mathcal{S}'([0,T]))$ \footnote{This set denotes the  space of linear maps between $\mathcal{T}$ and $\mathcal{S}'([0,T])$, the space of tempered distributions over $[0,T]$.} which is given by 
\begin{equation}\label{eq:rough_model2}
\begin{gathered}
(\Pi^{\bX}_s{\color{blue}\mathbf{1}})(r)=1\,, \quad (\Pi^{\bX}_s{\color{blue}X})(r)= X_r-X_s\,,\quad (\Pi^{\bX}_s{\color{blue}\hat{X}})(r)= \hat{X}_r-\hat{X}_s\,,\quad  (\Pi^{\bX}_s{\color{blue}\XX})(r) =\XX_{sr} \\ (\Pi^{\bX}_s{\color{blue}\Xi})(\psi)=\dot{X}(\psi) \,, \quad (\Pi^{\bX}_s{\color{blue}\hat{\Xi}})(\psi)= \dot{\hat{X}}(\psi)\,,\quad (\Pi^{\bX}_s{\color{blue}\hat{X}\Xi})(\psi)=\dot{\XX}_{s}(\psi)\,,
\end{gathered}
\end{equation}
where $\dot{\hat{X}}$, $\dot{\hat{X}} $ $\dot{\XX}_{s}$ are respectively the distributional derivatives of $X$, $\hat{X}$ and $ r\to \XX_{s,r}$ and $\psi\colon (0,T)\to \bR$ is a generic test function. The second one is a two-parameters map $\Gamma^{\bX}\colon[0,T]^2\to G$ which is defined by 
\begin{equation}\label{def:Gamma_rough}
\Gamma^{\bX}_{t,s}= \Gamma_{h_{st}}\,,\quad h_{st}= (X_t-X_s, \hat{X}_t- \hat{X}_s, \XX_{s,t})\,.
\end{equation} 
Following this correspondence, we can also describe every element $(Y, Y')\in\D_{\hat{X}}^{\gamma+\beta}([0,T])$ or $Y\in \mathcal{C}^{\alpha}([0,T])$ in terms of  \emph{modelled distributions} over $\Pi^{\bX}$, see \cite[Lemma 13.20]{FH2020} and \cite[Definition 3.1]{Hairer2014} for the general definition. It turns out that the same relation can be used  to describe singular controlled rough paths in terms of \emph{singular modelled distributions}, see \cite[Chapter 6]{Hairer2014}. To define them, we use the shorthand notation $|\cdot|_{l}$ for  the absolute value of the $l$-th component in $\mathcal{T}$.}
\begin{definition}
Let $\delta>0$,  $\eta\in \mathbf{R}$, and $\mathbf{X}\in\mathscr{C}([0,T])$. A function $\mathbf{Y}\colon (0,T]\to \bigoplus_{l<\delta}\mathcal{T}_l$ is said to be a singular modelled distribution, in symbols $\mathbf{Y} \in \mathcal{D}^{\delta, \eta}((0,T])$, if it satisfies
\begin{equation}\label{def:norm_Dgamma}
\|\mathbf{Y} \|_{\mathcal{D}^{\delta, \eta}}:=\sup_{l<\delta}\sup_{0<s\leq T}\frac{|\mathbf{Y}_s|_l}{(s\wedge 1)^{(\eta -l)\wedge 0}}  +\sup_{l<\delta}\sup_{\substack{0<s<t\leq T\\ |t-s|\leq s}} \frac{|\mathbf{Y}_t -\Gamma_{t,s}\mathbf{Y}_s|_{l}}{(s\wedge 1)^{\eta -\delta}|t-s|^{\delta- l}}<\infty\,.
\end{equation}
\end{definition}
Following the trivial estimate 
\begin{equation}\label{trivia_minimum}
(s\wedge 1)\leq s\leq (T\vee 1)(s\wedge 1)
\end{equation}
for any $0<s<T$, we can easily see that  every function of $\mathcal{C}^{\alpha, \eta}((0,T])$ is trivially included in $\mathcal{D}^{\alpha, \eta}((0,T])$. More generally, one can provide a full characterisation of singular controlled rough paths in terms of singular modelled distributions.
\begin{theorem}\label{thm:controlled_RP}
Let $\bX\in  \Cr([0,T])$ and $(Y, Y')\in \D_{\hat{X}}^{\gamma+\beta,\eta}((0,T])$. Then the function 
\begin{equation}\label{eq:controlled_RP}
\mathbf{Y}_t:= Y_t{\color{blue}\mathbf{1}}+ Y'_t{\color{blue}\hat{X}}
\end{equation}
is an element of $\mathcal{D}^{\gamma+ \beta, \eta}((0,T])$. On the other hand, for every not-integer $\eta\leq  \gamma+ \beta$ and any function  in $\mathcal{D}^{\gamma+ \beta, \eta}((0,T])$ with values in $\mathcal{T}_0 \oplus \mathcal{T}_{\beta}$, its two components are an element of $\D_{\hat{X}}^{\gamma+\beta,\eta}((0,T])$. Moreover there exist a constant $C>0$ depending on $T, \alpha, \beta, \gamma$ such that
\begin{equation}\label{eq:first_bound}
\frac{1}{C}\|\mathbf{Y} \|_{\mathcal{D}^{\gamma+\beta, \eta}}\leq  |Y_T|+ |Y'_T|+ \|Y,Y'\|_{\gamma+ \beta, \eta}\leq C\|\mathbf{Y} \|_{\mathcal{D}^{\gamma+\beta, \eta}}\,.
\end{equation}
{\color{black}By replacing $\beta$ with $\alpha$ and ${\color{blue}\hat{X}}$ with ${\color{blue}X}$, we can equivalently describe $\D_{X}^{\gamma+\alpha,\eta}((0,T])$ with $\mathcal{D}^{\gamma+ \alpha, \eta}((0,T])$.}
\end{theorem}
\begin{proof}
Let us prove the left-hand inequality in \eqref{eq:first_bound}. We decompose  $\|\mathbf{Y} \|_{\mathcal{D}^{\gamma+\beta, \eta}}$ as the sum $\|\mathbf{Y} \|_{\mathcal{D}^{\gamma+\beta, \eta}}= \mathbf{Y}_1+ \mathbf{Y}_2$. Using the definition of $\D_{\hat{X}}^{\gamma+\beta,\eta}$, the definition of $\Gamma_{s,t}$ in \eqref{def:Gamma_rough} and the trivial estimate \eqref{trivia_minimum} one has trivially $ \mathbf{Y}_2\leq \|Y,Y'\|_{\gamma+ \beta, \eta}$. In addition, for any  $0<s\leq T$ we have
\[\frac{|Y'_s|_l}{(s\wedge 1)^{(\eta -l)\wedge 0}}\leq  T^{\gamma}\|Y,Y'\|_{\gamma+ \beta, \eta}+ T^{\gamma-\eta}|Y'_T|\,, \quad \frac{|Y_s|_l}{(s\wedge 1)^{(\eta -l)\wedge 0}}\leq T^{\gamma+ \beta}\|Y,Y'\|_{\gamma+ \beta, \eta}+  T^{\gamma+ \beta-\eta}|Y'_T|+  T^{\gamma-\eta}|Y_T|\,.\]
Thereby obtaining the first part of \eqref{eq:first_bound}. On the other hand, Let $Z$ and $Z'$ denote the two component of  $\mathbf{Z}$. Thanks to Lemma \ref{lem:equality} with $\alpha= \gamma$ and $\delta= \gamma+\beta$ there exists a constant $D>0$ such that one has
\[ \sup_{0 < s < t \le T} \frac{|Z'_t - Z'_s|}{s^{\eta -(\gamma+ \beta)}|t-s|^{\gamma}}\leq D \|\mathbf{Z} \|_{\mathcal{D}^{\gamma+ \beta, \eta}}\,.\]
Combining this inequality with \eqref{trivia_minimum} and the trivial estimate
\[|Z_T|+ |Z'_T|\leq \frac{2}{T^{-(\eta\wedge 0)}\vee T^{-((\eta-\beta)\wedge 0)}}\|\mathbf{Z} \|_{\mathcal{D}^{\gamma+ \beta, \eta}}\,, \]
we obtain the right-hand inequality in  \eqref{eq:first_bound}, as long as we are able to estimate $R_{s,t}= Z_t-Z_s- Z'_s(\hat{X}_t-\hat{X}_s)$. Iterating the trivial identity $R_{s,t}= R_{s,u}+R_{u,t} + (Z'_u- Z'_s)(\hat{X}_t-\hat{X}_u) $, we can achieve this estimate by repeating  the procedure in the proof of Lemma \ref{lem:equality}. {\color{black}Thus for any  $0<s<t\leq T$ and $\eta\leq  \gamma +\beta$ we consider $N=\lfloor\log_2(t/s)\rfloor$ and there exist two constants $C',C''>0$ depending on $\eta$, $\gamma$ and $\beta$ such that
\[\begin{split}
|R_{s,t}|&\leq |R_{2^Ns,t}| +\sum_{n=0}^{N-1}|R_{2^ns ,2^{n+1}s}|+ \sum_{n=0}^{N-1}|Y'_{ 2^{n+1}s}- Y'_{2^ns}||\hat{X}_{t}-\hat{X}_{s_{n+1}}|\\&\leq  \|\mathbf{Y} \|_{\mathcal{D}^{\gamma+ \beta, \eta}}s^{\eta- (\gamma+ \beta)}\left((t-2^Ns)^{\gamma+ \beta}+ C'(t-s)^{\gamma+ \beta}+ C''\|\hat{X}\|_{\beta}(t-s)^{\beta}(t-s)^{\gamma} \right)\\&\leq (1+ C'+ C'')\|\mathbf{Y} \|_{\mathcal{D}^{\gamma+ \beta, \eta}}s^{\eta- (\gamma+ \beta)}(t-s)^{\gamma+ \beta}\\&\leq (1+ C'+ C''\|\hat{X}\|_{\beta})\|\mathbf{Y} \|_{\mathcal{D}^{\gamma+ \beta, \eta}}(s\wedge 1)^{\eta- (\gamma+ \beta)}(t-s)^{\gamma+ \beta}\,.
\end{split}\]}
Thereby obtaining the second part of bound \eqref{eq:first_bound}. {\color{black}The equivalence among $\D_{X}^{\gamma+\alpha,\eta}((0,T])$ and  $\mathcal{D}^{\gamma+ \alpha, \eta}((0,T])$ follows immediately by replacing $\beta$  and ${\color{blue}\hat{X}}$ as in the statement}.
\end{proof}
{\color{black}
\begin{remark}
Using this equivalence, the general tools of regularity structures can provide an alternative proof of Theorem \ref{thm:improper_young} and \ref{thm:improper_rough}  when $\eta_2$ is respectively $\alpha$ and $\alpha+ \beta$. We will give a simple sketch of it in this second case. Starting from a singular controlled rough path $(Y,Y')\in\D_{\hat{X}}^{\gamma+\beta,\eta_1}((0,T])$, we consider $\mathbf{Y}\in \mathcal{D}^{\gamma+ \beta, \eta_1}((0,T])$ given  in \eqref{eq:controlled_RP} and we define the map  $\mathbf{Y} {\color{blue}\Xi} \colon (0,T]\to \mathcal{T}$ as
\[(\mathbf{Y} {\color{blue}\Xi})_t:= Y_t{\color{blue}\Xi}+ Y'_t{\color{blue}\hat{X}\Xi}\,.\]
This function is an example of product between two singular  modelled distribution  and  it is an element of $\mathcal{D}^{\gamma+ \beta +\alpha-1, \eta_1+ \alpha-1}((0,T])$ (see \cite[Prop 6.12]{Hairer2014}). Writing down the hyphotesis of  Theorem  \ref{thm:improper_rough} when $\eta_2= \alpha+\beta$, we have that $\eta_1+\alpha-1>-1$. This condition, together with the assumption $\alpha+\beta+ \gamma-1>0 $ allows us to apply the \emph{reconstruction theorem} for singular modelled distribution \cite[Prop.  6.9]{Hairer2014}. Loosely speaking, the theorem states that for any $\mathbf{Z}\in \mathcal{D}^{\delta, \eta}((0,T]) $ satisfying $\delta>0$ and $\eta \wedge \alpha-1>-1$ we can uniquely associate a distribution $ \mathcal{R}(\mathbf{Z})\in \mathcal{S}'([0,T])$ which satisfies a bound of the type
\begin{equation}\label{eq:reconstruction} \mathcal{R}(\mathbf{Z})(\psi^{\lambda}_s) = \Pi^{\bX}_s (\mathbf{Z}_s)(\psi^{\lambda}_s)+ \OO(\lambda^{\delta})\,,
\end{equation}
where $\psi^{\lambda}_s=\lambda^{-1}\psi((\cdot- s)/\lambda)$, $\psi $ is a generic smooth, compactly supported function $\psi\colon (-1, 1)\to \bR$ and $\lambda$ is sufficiently small. Writing this general result in the case of $\mathbf{Y} {\color{blue}\Xi}$, we obtain then a unique distribution $\mathcal{R}(Y {\color{blue}\Xi})$ such that
\begin{equation}\label{eq:reconstruction2}
\mathcal{R}(Y {\color{blue}\Xi})( \psi^{\lambda}_s)= Y_s\dot{X}(\psi^{\lambda}_s)+ Y'_s\dot{\XX}_s( \psi^{\lambda}_s)+ \OO(\lambda^{\gamma+ \beta +\alpha-1})\,.
\end{equation}
Comparing \eqref{eq:reconstruction2} with \eqref{eq:sewing_rough}, we realise that $\mathcal{R}(Y {\color{blue}\Xi})$ is the distributional derivative of the improper rough integral and one has the identity up to addition of constants
\begin{equation}\label{eq:recon_rough}
\int^t_{0^+}Y_rd\bX_r= \int_0^t \mathcal{R}(Y {\color{blue}\Xi})(ds)\,,
\end{equation}
where the right-hand side of \eqref{eq:recon_rough} is the  \emph{primitive} of  $\mathcal{R}(Y {\color{blue}\Xi})$ that cancels in zero. To define this operation, it sufficient to test $\mathcal{R}(Y {\color{blue}\Xi})$ against a sequence of smooth functions approximating the indicator of the set $[0,t]$ and to show that there is a unique limit. The additional properties of the improper rough integral can be deduced by studying the analytic regularisation of taking the primitive of a distribution.
\end{remark}
}

\section{Application to Fractional Modelling and Rough Volatility}

Recent advances in quantitative finance led to models where (stochastic) volatility {\color{black} runs on ``rougher'' than diffusive scales}, locally described by a Hurst parameter $H \in (0,1/2)$. Following \cite{bayer2016pricing,bayer2019regularity}, moves in log-price then involve stochastic It\^o integrals of the form
\begin{equation} \label{e:roughvol2}
          \int_0^T f (W^H_t) d W_t \;,
\end{equation}
where  $f$ is a sufficiently smooth map, $W$ is a standard Brownian motion, which we assume two-sided so that $\xi = \dot{W}$ is white noise on $\R$. L\'evy or Riemann--Liouville
fractional Brownian motion is then given by
\begin{equation} \label{equ:fishyFBM}
       W^H_t := \int_0^t K^H(t-s)\, dW_s  = (K^H * \xi_+)(t)
\end{equation}
with Riemann--Liouville kernel $K^H(u) = u^{H-1/2}1_{u>0}$, and $\xi_+$ white noise  on $\R_+$, given as distributional derivative of $W1_{\R_+}$. There is interest e.g. from asymptotic option pricing \cite{bayer2019regularity,friz2018precise}) to have a rough path type stability for this integral. For $H < 1/2$, this requires an extension of rough paths that constitutes a most instructive example of a non-trivial regularity structures \cite{bayer2019regularity}. In order to apply the standard results of the general theory developed  in \cite{Hairer2014}, it is desirable -- and necessary for BPHZ renormalisation \`a la \cite{ch2016analytic} -- to replace the process \eqref{equ:fishyFBM} by a stationary process, which we may take of the form
\begin{equation} \label{equ:}
  \widehat{W}^H_t := (\widehat{K}^H * \xi)(t) \,,
\end{equation}
where $\xi$ is a white noise on $\R$ and $\widehat{K}^H$ is a smooth function $\widehat{K}^H\colon \R\setminus\{0\} \to \R_+$ satisfying   the properties:
\begin{enumerate}
\item $\widehat{K}^H \equiv K^H$ on $[0,T]$ and $\supp(\widehat{K}^H)\subset [0,2T]$;
\item there exist a constant $C>0$ such that for $k=0,1,2$ one has
\[|\d^k \widehat{K}^H(u)|\leq C|\d^k K^H(u)|\,.\]
\end{enumerate}
The existence of such function is a trivial exercise in case of the Riemann--Liouville kernel. We now show that $W^H$ can be regarded as {\it singular} controlled rough path when $H>1/4$. The reference inhomogeneous rough path $\mathbf{W}$, over stationary noise, is obtained  {\color{black} by considering} the triplet of functions
\begin{equation}\label{defn:rough_vol_model}
\mathbf{W}=(W_t\,, \,  \widehat{W}^H_t\,,\, \int_{s}^t(\widehat{W}^H_r-\widehat{W}^H_s)dW_r)\,,
\end{equation}
where the integral in $dW$ is an It\^o integral. Using some standard tools of stochastic calculus one has  that the conditions \eqref{defn:rough_vol_model} defined an  element $\mathbf{W}$ which belongs a.s. to $\Cr([0,T])$, where $\alpha=1/2^-$ and $\beta=H^-$. Before stating a rigorous result on $W^H$, we recall an elementary property of the function
\begin{equation}
t\mapsto \int_{-\infty}^0\widehat{K}^H(t-r)\,d B_r\,, \end{equation}
which is smooth on $(0,\infty)$ for any $B\in\CC^{\delta}([-2T,0])$, $\delta\in(0,1)$ and we can {\color{black} differentiate} under the sign of integral, because $\widehat{K}^H$ is smooth outside zero.
\begin{lemma}\label{lem:R bound} Let $\delta\in(0,1)$ and $B\in\CC^{\delta}([-2T,0])$ such that $B_0=0$. There exists a constant $M>0$ depending on $T$ and $\delta$ such that for any $t>0$ one has 
\begin{equation}\label{eq:R bound}
\Big|\int_{-\infty}^0\d_t\widehat{K}^H(t-r)\,dB_r\Big|\leq M t^{H-3/2+\delta} \|B\|_{\delta; [-2T,0]}\,  .
\end{equation}
\end{lemma}
\begin{proof}
First consider the case when $B$ is a $\mathcal{C}^{1}([-2T, 0])$ function. Notice that the property (2) of $\widehat{K}^H $ implies the trivial bound
\begin{equation}\label{eq:first_site}
\Big|\int_{-\infty}^0\d_t\widehat{K}^H(t-r)\,dB_r\Big|\leq C\int_0^{2T}|\partial_tK^H(t+r)|dr\,\|B\|_{1; [-2T,0]}
\leq  C 't^{H-1/2}\|B\|_{1; [-2T,0]}\,,
\end{equation}
for some constant $C'>0$ depending on $T$. Next, by integration by parts one sees that
\[
\int_{-\infty}^0\widehat{K}^H(t-r)\,d B_r=-\int_0^{2T} B_{-r}\partial\widehat{K}^H(t+r)\,dr\,.
\]
Therefore the property (2) of $\widehat{K}^H $ implies there is also has another constant $C''>0$ such that
\begin{equation}\label{eq:second_site}
\Big|\int_{-\infty}^0\d_t\widehat{K}^H(t-r)\,dB_r\Big|\leq C\int_0^{2T}|\partial^2_tK^H(t+r)|dr\sup_{-2T\leq r\leq 0}|B_{r}|
\leq  C ''t^{H-3/2}\sup_{-2T\leq r\leq 0}|B_{r}|\,.
\end{equation}
In case when $\delta\in (0,1)$, we consider a standard  sequence of mollifiers $\phi_{\varepsilon}$ on $[0,T]$ such that  $\int \phi_{\varepsilon} =1$ and defining $B^{\varepsilon}= \phi_{\varepsilon}* B$ we can combine \eqref{eq:first_site} and \eqref{eq:second_site} to obtain
\begin{equation}\label{eq:third_site}
\begin{split}
\Big|\int_{-\infty}^0\d_t\widehat{K}^H(t-r)\,dB_r\Big|&\leq \Big|\int_{-\infty}^0\d_t\widehat{K}^H(t-r)\,d(B_r- B^t_r)\Big|
+ \Big|\d_t\int_{-\infty}^0\d_t\widehat{K}^H(t-r)\,dB^t_r\Big|\\&\leq C''t^{H-3/2}\sup_{-2T\leq r\leq 0}|B_{r}- B^t_r|+C't^{H-1/2}\|B^t\|_{1; [-2T,0]} \,.
\end{split}
\end{equation}
Using the properties of the mollifiers we can easily show that there exists a constant $C'''>0$ such that
\begin{equation}\label{eq_fourth_site}
\sup_{-2T\leq r\leq 0}|B_{r}- B^t_r|\leq C'''t^{\delta}\|B\|_{\delta; [-2T,0]}\,, \quad \|B^t\|_{1; [-2T,0]}\leq C'''t^{\delta-1}\|B\|_{\delta; [-2T,0]}\,.
\end{equation}
Combining \eqref{eq:third_site} and \eqref{eq_fourth_site}, we obtain the desired inequality.
\end{proof} 
Using this path-wise estimate we can easily show the following description.
\begin{proposition}
The couple $(Y,Y') = (W^H, 1)$ belongs a.s. to $\D_{\widehat{W}^H}^{\gamma+H,H^-}((0,T])$ for any $\gamma\in (0,1)$ such that $\gamma+ H<1$.
\end{proposition}
\begin{proof}
We prove the result when $\gamma + H =1$ and the general case $\gamma+H<1$ follows by the embedding \eqref{eq:trivial_properties}. Since $Y'$ is constant we need only to estimate $R^{Y}$. It follows immediately from property (1) of $\widehat{K}^H$ that $ W^H=\widehat{K}^H*\xi_+$ on $[0,T]$. By writing $\xi= \xi_++ \xi_{-}$ where $\xi_-$ is a white noise on $(-\infty, 0]$ for any for $0<s<t<T$, we have the a.s. identity 
\begin{equation}\label{eq:diff_frac}
R_{s,t} = W^H_t - W^H_s -( \widehat{W}^H_t-\widehat{W}^H_s) = \int_{-\infty}^0\widehat{K}^H(t-r)- \widehat{K}^H(s-r)\,dW_r \,.
\end{equation}
We apply Lemma \ref{lem:R bound} with $B=W$ and $\delta=1/2^{-}$. Since $\|W\|_{\delta, [-2T, 0]}$ is almost surely finite, \eqref{eq:R bound} yields the bound
\begin{equation}
|R_{s,t}|\leq |t-s|\sup_{v\in[s,t]}|\d_u\Big(\int_{-\infty}^0\widehat{K}^H(u-r)\,dB_r\Big)|_{u=v}\leq D |t-s|s^{H^{-}-1}\,.
\end{equation}
which shows the result.
\end{proof}
Since we want to describe the rough integration of a non linear function of  $(W^H, 1)$ we fix $\gamma= (1/4)^+$ so that $(1/2)^- + H +\gamma>1 $. Applying then some general tools of regularity structures, we can actually describe the stochastic integral \eqref{e:roughvol2} in terms of some smooths approximants. In what follows, we introduce $W^\eps$, $\widehat{W}^{\eps,H}$ and $W^{\eps,H}$ as smooth approximations of $W$,  $\widehat{W}^H $ and $W^H$, built by replacing the white noise $\xi$ by a stationary approximation $\xi^\eps$, obtained by convolution with a rescaled mollifier function, $\varrho_{\varepsilon}= \eps^{-1}\varrho(\eps^{-1}\cdot)$.
We also introduce the canonical stationary rough path
$$
\mathbf{W}^\eps=(W^\eps_t\,, \,  \widehat{W}^{\eps,H}_t\,,\, \int_{s}^t(\widehat{W}^{\eps,H}_r-\widehat{W}^{\eps,H}_s)dW^{\eps}_r)\,,
$$
and its BPHZ renormalisation, see \cite[Chap 14]{FH2020}
$$
\mathbf{W}^{\eps;\mathrm{ren}}=(W^\eps_t\,, \,  \widehat{W}^{\eps,H}_t\,,\, \int_{s}^t(\widehat{W}^{\eps,H}_r-\widehat{W}^{\eps,H}_s)dW^{\eps}_r - c^{\eps,H} (t-s))\,,
$$
best understood as integrated effect of centring in mean $\widehat{W}^{\eps,H}_t \xi_t$. The renormalisation constant 
\begin{equation}\label{ren_constant}
c^{\eps,H}=\mathbf{E} ( \widehat{W}^{\eps,H}_t \xi^{\eps}_t) \;,
\end{equation}
does not dependent on $t$ by stationarity. 
\begin{theorem}\label{final_thm}
For $f \in \CC_b^3$ and $H> 1/4$ one has the convergence in probability {\color{black} uniformly on compact sets
\[
\left(\int_0^t f (W^{\eps,H}_s) \xi^{\eps}_sds - c^{\eps,H} \int_0^t Df (W^{\eps,H}_s) ds \right)_{t\in[0,T]}\to \left(\int_0^t f ({W}^H_s) d W_s\right)_{t\in[0,T]} \;, 
\]
where $(c^{\eps,H})_{\eps > 0}$ diverges at rate $\eps^{H-1/2}$.}
\end{theorem} 
\begin{proof}
We split the proof into two parts: First of all we show that for any $f \in \CC^3_b$ and $H> 1/4$ the couple $F(W^H_t):=(f (W^H_t), Df (W^H_t))$ is still a singular controlled rough path and one has the a.s. identity
\begin{equation}\label{identity_reconstruct}
\int^t_{0^+}F(W^H_s)d\mathbf{W}_s= \int_0^t f ({W}^H_s) d W_s\,.
\end{equation}
The first property on $F(W^H_t)$ is a variation on standard results of composition of controlled rough paths \cite[Lem. 7.3, Thm 7.5]{FH2020} and easy to show by hand. To establish the identity  \eqref{identity_reconstruct} we first note, by straightforward adaptation of \cite[Prop. 5.1]{FH2020}, that the a.s. identity
\begin{equation}\label{identity_reconstruct2}
\int^t_{s}F(W^H_r)d\mathbf{W}_r= \int_s^t f ({W}^H_r) d W_r\,,
\end{equation}
holds for any $0<s<t$. By Theorem \ref{thm:improper_rough}, the left-hand side of \eqref{identity_reconstruct2} converges to the left-hand side of \eqref{identity_reconstruct}. 
Once we have the identity \eqref{identity_reconstruct} we can recall the path-wise identity \eqref{eq:recon_rough} to have
\begin{equation}\label{almost_last}
\int_0^t f ({W}^H_s) d W_s= \int_0^t \mathcal{R}_{\mathbf{W}}(F(W^H){\color{blue}\Xi})(ds)\,,
\end{equation}
where $\mathcal{R}_{\mathbf{W}}$ is the reconstruction operator associated to the model $\Pi_{\mathbf{W}}$ {\color{black} and $F(W^H){\color{blue}\Xi}$ is given by
\begin{equation}\label{modelled}
(F(W^H){\color{blue}\Xi})_t:= f (W^H_t){\color{blue}\Xi}+ Df (W^H_t){\color{blue}\hat{X}\Xi}\,.
\end{equation}}
Starting from \eqref{almost_last}, in the second part of  the proof we apply the result contained in  \cite{Hairer2014, ch2016analytic} to obtain that this reconstruction operation is the following limit in probability 
\begin{equation}\label{convergence_reconstruction}
{\mathcal{R}}^{\mathrm{ren}}_{\mathbf{W}^{\eps}}(F(W^{\eps,H}){\color{blue}\Xi})\to\mathcal{R}_{\mathbf{W}}(F(W^H){\color{blue}\Xi})\,,
\end{equation}
where ${\mathcal{R}}^{\mathrm{ren}}_{\mathbf{W}^{\eps}}={\mathcal{R}}_{\mathbf{W}^{\eps;\mathrm{ren}}}$ is a sequence of smooth functions {\color{black}and $F(W^{\eps,H}){\color{blue}\Xi}$ is defined like in \eqref{modelled} with $W^{\eps,H}$ instead that $W^{H}$. The limit in probability in  \eqref{convergence_reconstruction} is taken by looking at  ${\mathcal{R}}^{\mathrm{ren}}_{\mathbf{W}^{\eps}}(F(W^{\eps,H}){\color{blue}\Xi})$ and $\mathcal{R}_{\mathbf{W}}(F(W^H){\color{blue}\Xi})$ as random variables belonging to some negative H\"older space $\mathcal{C}^{\delta}([0,T])$ \cite[Def 3.7]{Hairer2014} where $\delta\in (-1,0)$. Since the operation of taking a primitive is a continuous operator  between $ \mathcal{C}^{\delta}([0,T])$ and continuous function over $[0,T]$, the convergence \eqref{convergence_reconstruction} implies the convergence in probability
\begin{equation}\label{convergence_reconstruction2}
\left(\int_0^t{\mathcal{R}}^{\mathrm{ren}}_{\mathbf{W}^{\eps}}(F(W^{\eps,H}){\color{blue}\Xi})(ds)\right)_{t\in [0,T]}\to\left(\int_0^t\mathcal{R}_{\mathbf{W}}(F(W^H){\color{blue}\Xi})(ds)\right)_{t\in [0,T]}\,,
\end{equation}
uniformly on compact set. To identify the values of the approximating sequence in \eqref{convergence_reconstruction2}, we use again \eqref{identity_reconstruct} and we calculate the rough integration of {\color{black}$F(W^{\eps,H}_t):=(f (W^{\eps,H}_t), Df (W^{\eps,H}_t))$ against $\mathbf{W}^{\eps;\mathrm{ren}}$}. It is clear that the second level perturbation which distinguishes these models leads to second order correction, which manifests itself as
\begin{equation}\label{explicit_renorm}
\int_0^t{\mathcal{R}}^{\mathrm{ren}}_{\mathbf{W}^{\eps}}(F(W^{\eps,H}){\color{blue}\Xi})(ds)=\int^t_{0}F(W^{\eps,H}_s)d\mathbf{\mathbf{W}}_s^{\eps;\mathrm{ren}}=\int_0^t f (W^{\eps,H}_s) \xi^{\eps}_sds - c^{\eps,H} \int_0^t Df (W^{\eps,H}_s) ds\,.
\end{equation}}
Combining the convergence \eqref{convergence_reconstruction} with \eqref{explicit_renorm} we obtain the result. To compute the renormalisation constant from \eqref{ren_constant} we use the It\^o isometry obtaining
\[c^{\eps,H}= \int_{\R} \int_{\R} \widehat{K}^H(t-s)\varrho_{\eps}(s-r)\varrho_{\eps}(t-r)drds= \int \widehat{K}^H(-s)\bar{\varrho}_{\eps}(s)ds\,,\]
where $\bar{\varrho}(y):=(\varrho(-\,\cdot)*\varrho) (y)$. Since $ \widehat{K}^H= K^H$ on a neighbourhood of $0$ we can easily assume that the support of $\bar{\varrho}$ is contained in it by taking $\eps$ sufficiently small. Therefore we obtain
\[c^{\eps,H}= \eps^{-1}\int_{\R}\widehat{K}^H(-s)\bar{\varrho}(\eps^{-1}s)ds= \eps^{-1}\int_{\R} K^H(s)\bar{\varrho}(\eps^{-1}s)ds=\eps^{H-1/2}\int_{0}^{+\infty}K^H(s)\bar{\varrho}(s)ds\,.\]
\end{proof}

{{\bf Acknowledgements:} CB was supported by DFG Research Unit FOR2402. PKF has received funding from the European Research Council (ERC) under the European Union's Horizon 2020 research and innovation programme (grant agreement No. 683164) and the DFG Excellence Cluster MATH+ (Projekt AA4-2). MG thanks the support of the Austrian Science Fund (FWF) through the Lise Meitner programme M2250-N32.}

%
\bibliographystyle{abbrv}
\bibliography{bibliography}

\end{document}